\documentclass[reqno, a4paper]{amsart}

\usepackage{txfonts} 

\usepackage{amsmath, amssymb, bbm, amsthm,enumitem, hyperref, color}
\usepackage{comment}

\DeclareMathOperator{\Tr}{Tr}

\newtheorem{thm}{Theorem}[section]

\newtheorem{definition}{Definition}[section]
\newtheorem{proposition}{Proposition}[section]
\newtheorem{corollary}{Corollary}[section]
\newtheorem{remark}{Remark}[section]
\newtheorem{lemma}{Lemma}[section]
\numberwithin{equation}{section}

\newcommand{\pa}{\partial}

\newcommand{\R}{\mathbb R}
\newcommand{\E}{\mathbb E}
\newcommand{\teta}{\tilde \eta}

\DeclareMathOperator*{\argmax}{arg\,max}

\thanks{JB is supported by the FWF grant DOI 10.55776/P36835. ZW is supported by NSFC Grant No.12141105.}

\title{Exciting games and Monge-Amp\`{e}re equations}

\author{Julio Backhoff}
\address{ Department of Mathematics, University of Vienna}
\email{julio.backhoff@univie.ac.at}

\author{ZhiZhang Wang} 
\address{School of Mathematical Science, Fudan University, Shanghai, China}
\email{zzwang@fudan.edu.cn}

\author{Xin Zhang} 
\address{ Department of Finance and Risk Engineering, New York University}
\email{xz1662@nyu.edu}

\date{}

\begin{document}

\begin{abstract}

We consider a competition between $d+1$ players, and aim to identify the ``most exciting game'' of this kind. This is translated, mathematically, into a stochastic optimization problem over  martingales that live on the $d$-dimensional subprobability simplex $\Delta$ and terminate on the vertices of $\Delta$ (so-called \emph{win-martingales}), with a cost function related to a scaling limit of Shannon entropies. We uncover a surprising connection between this problem and the seemingly unrelated field of \emph{Monge-Amp\`{e}re} equations: If $g$ solves
\begin{equation*}
\begin{cases}
g(x)=\log \det\left(\frac{1}{2}\nabla^2 g(x)\right), \quad  \, \, \,  \, x \in \Delta, \\
g(x)=\infty, \quad \quad \quad \quad  \ \ \ \ \ \quad \quad  \, \ \ \ \ \ x\in \partial \Delta,
\end{cases}
\end{equation*}
then the winning-probability of the players in the most exciting game is given by   $$dM_s=\sqrt{\frac{2 (\nabla^2 g(M_s))^{-1}}{1-s} } \, dB_s.$$ To formalize this, a detailed quantitative analysis of the Monge-Amp\`{e}re equation for $g$ is crucial. This is then leveraged to prove that $M$ is indeed an optimal win-martingale.

\end{abstract}

\maketitle

\section{Introduction}

This paper is inspired by the question \[\text{``What is the most exciting game?"},\] formulated by David Aldous  \cite{Al22b,Al22a}. In this problem one seeks to determine the optimal strategy for revealing information in a game, in order to capture and sustain the attention of an audience. This is relevant not only in the entertainment / sports industry but also in fields such as behavioral economics and education \cite{El17,JAE15}. Let us first introduce a mathematical framework which allows to explore the above question. 

Suppose there is a game (or an election, etc.) between $d+1$ players during the time period $[t,1]$. Assume the dynamics of the winning probability of  player $i$ is represented by a stochastic process $(M^i_s)_{s \in [t,1]}$, $i=0,\dotso, d$. It is clear that $\sum_{i=0}^d M^i_s=1$, so the winning probability at time $s$ of all players is fully characterized by the $d$-vector $M_s:=(M^1_s, \dotso, M_s^d)$. Denote the interior of the state space of $M$ by
\[\textstyle\Delta:=\left\{(x_1,\dotso,x_d) \in \mathbb{R}^d: \, x_i > 0, \, i=1,\dotso,d, \, \sum_{i=1}^d x_i < 1 \right\}, \] 
and its set of vertices by $$\mathcal{V}:=\{0,e_1,\dotso,e_d \}.$$
For prediction markets where people can bet on the outcome of the game, the stochastic process $(M_s)_{s \in [t,1]}$ is usually modeled as a martingale, owing to market efficiency assumptions such as the no-arbitrage principle. At the terminal time $s=1$, we have $M_1 \in \mathcal{V}$ $a.s.$ since the outcome of the game would be revealed at its end, and hence one and only one of $\{M^i_1\}_{i=0,\dotso,d}$ is equal to $1$. { On the other hand, if the game starts at time $t$ then $M_t$ encodes the apriori winning probabilities of the players and is deterministic.} For these reasons, we identify a game in this paper with a so-called \emph{win-martingale}:

\begin{definition}\label{def:win_intro}
   A $d$-dimensional martingale $(M_s)_{s \in [t,1]}$ is a win-martingale on $[t,1]$ if $M_t \in \Delta$ is deterministic, $M_1 \in \mathcal{V}$ a.s., and the quadratic covariation of $(M_s)_{s\in[t,1]}$ is absolutely continuous. Denote by $\mathcal{M}^{win}_{t,x}$ the set of  win-martingales starting from  $x \in \Delta$ at time $t$.  
\end{definition}
We remark that win-martingales have been used as models for sports games between multiple players, political betting, as well as for double barrier financial options, and are also known as belief-martingales in the context of information design \cite{El17,JAE15,Zh22}. {\color{blue}\footnote{We thank R. Aid for bringing these relevant works in Economics to the attention of one of the authors.} }

With a mathematical model of a game at hand, we now explain how to measure its level of ``excitement''. For random variables, the usual notion of excitement or uncertainty is the Shannon entropy of its law, since it is known to quantify the level of unpredictability / novelty / information content within a system \cite{Be60,Sh48}. See however \cite{JAE15} for more possibilities. For stochastic processes in continuous time this is much more subtle, since there is no canonical reference measure playing the role of Lebesgue measure in such infinite-dimensional framework.  The criterion considered in this paper is based on the idea of taking a scaling limit of Shannon entropies for time-discretized processes, as suggested in \cite{Al22b,Al22a}. For two players (i.e.\ $d=1$), Beiglb\"{o}ck and the first author observed in \cite{BBexciting} that maximizing this scaling limit of entropies over win-martingales is related to maximizing the functional $\mathbb E [ \int_t^1  \log\det (\Sigma_s) \, ds]$ over $M \in \mathcal{M}_{t,x}^{win}$, where $(\Sigma_s)_{s \in [t,1]}$ is the instantaneous quadratic variation of $M$. We elaborate more on this in Section \ref{sec:lite} below.

In the present work we take the point of view of \cite{BBexciting} too, but in higher dimensions, i.e., when the game has an arbitrary finite number of players / outcomes. Accordingly, 
throughout this paper we will be concerned with the problem 
\begin{align}\label{eq:optimal}\textstyle 
v(t,x):= \inf_{M \in \mathcal{M}^{win}_{t,x}}\mathbb{E}\left[ \int_t^1  -\log \det(\Sigma_s) \, ds \right], \tag{OPT}
 \end{align}
 where we passed to a formulation with an infimum for notational convenience at later stages. The function $v:[0,1] \times \Delta \to \mathbb R\cup\{+\infty\}$ is then the so-called value function of the problem. Following the stochastic control playbook we can guess the  Hamilton-Jacobi-Bellman (HJB)  equation associated to this problem, and a few computations reveal that 
\begin{equation}\label{eq:HJB} \tag{HJB}
\begin{cases}
 -\pa_t v(t,x)=\log\det \left( \frac{1}{2} \nabla^2 v(t,x) \right)+d,  \hspace{14pt} (t,x) \in [0,1) \times \Delta,  \\
\quad  \, \ v(t,x)=+\infty, \quad  \hspace{80pt} (t,x) \in [0,1) \times \partial \Delta, \\
\quad \ \, v(t,x)=0, \quad \quad  \hspace{80pt}  (t,x) \in \{1\} \times \Delta.
\end{cases}
\end{equation}
In Proposition~\ref{prop:uniqueness}, we show that \eqref{eq:HJB} is uniquely solved by $v$ in the viscosity sense.

To further identify the optimizer of \eqref{eq:optimal}, we observe first, via a scaling argument, that $v(t,x)=(1-t)v(0,x)+d(1-t) \log(1-t)$. Together with \eqref{eq:HJB} this then shows that  the function $\Delta \ni x \mapsto v(0,x)$ satisfies the following equation  on the simplex 
\begin{equation}\label{eq:Monge} \tag{MA}
\begin{cases}
g(x)=\log \det\left(\frac{1}{2}\nabla^2 g(x)\right), \quad x \in \Delta, \\
g(x)=+\infty, \quad \quad \quad \quad  \ \  \quad \quad x \in \partial \Delta. 
\end{cases}
\end{equation}
This is an elliptic Monge-Amp\`{e}re equation with \emph{infinite boundary condition.} Monge-Amp\`{e}re equations are fundamental in optimal transport and differential geometry \cite{Fi17, Le24, CNS1}.   The well-posedness of the Dirichlet problem with regular boundary conditions has been established by Caffarelli, Nirenberg, and Spruck in \cite{CNS1}.
However, the absence of the strict convexity and smoothness of the domain $\Delta$ in \eqref{eq:Monge} creates challenges in analyzing the regularity and the boundary asymptotics of the equation, which are crucial for the construction of the optimal win-martingale. We overcome these difficulties by designing barrier functions and utilizing an invariant property of \eqref{eq:Monge}, i.e.,  scaled smooth solutions satisfy essentially the same equation.

Denoting the minimal eigenvalue of a symmetric matrix $A$ by $\lambda_{min}(A)$, we can state the first main result of this paper:
\begin{thm}\label{thm:quantitive}
A unique smooth solution $g$ to \eqref{eq:Monge} exists, and the following estimates hold
\begin{itemize}
    \item[(i)] $\lim\limits_{x \to x^0 } \lambda_{min}( \nabla^2 g(x)) < +\infty, \quad \text{for $x^0  \in \partial \Delta \backslash \{0,e_1,\dotso,e_d\}$} $;
    \item[(ii)] $\sup_{x\in\Delta } \nabla g(x)^{\top} (\nabla^2g(x))^{-1} \nabla g(x)<\infty $. 
\end{itemize}
\end{thm}

We stress that the novelty of Theorem \ref{thm:quantitive} lies in (i)-(ii) (especially (i)), and that these are also the most challenging parts of it. { Although the existence of a smooth solution has been proved by \cite{GuJi04}, we provide an alternative argument to make the paper self-contained.}

With the regularity and estimates of $g$ at hand, going back to the HJB Equation we identify the candidate optimal martingale as the solution to the stochastic differential equation 
\begin{align}\label{eq:AMTG}\textstyle
dM_s=\sqrt{\frac{2 (\nabla^2 g(M_s))^{-1}}{1-s} } \, dB_s\,,  \tag{AM}
\end{align}
where $(B_s)$ is a $d$-dimensional Brownian motion. In tandem with \cite{BBexciting} we refer to $M$ as the \emph{Aldous Martingale}. The well-posedness of \eqref{eq:AMTG} and the fact that $M_1 \in \mathcal{V}$ $a.s.$, require significant work and heavily rely on the boundary asymptotics obtained in Theorem~\ref{thm:quantitive}. This eventually leads us to the  second main result of the paper:

\begin{thm}\label{thm:intro_existence}
For any $x\in\Delta$, we have that 
\begin{itemize}
    \item[(i)] There exists a strong solution to \eqref{eq:AMTG} starting from position $x$ at time $t$;
    \item[(ii)] $(M_s)_{s \in [t,1]}$ is a win-martingale. 
\end{itemize}
Furthermore, the Aldous martingale \eqref{eq:AMTG} is the unique (in law) optimizer of \eqref{eq:optimal}.
\end{thm}

We are not aware of an explicit expression for $g(\cdot)=v(0,\cdot)$, other than for $d=1$. This makes Theorem \ref{thm:intro_existence} challenging, since the original one-dimensional arguments put forward in \cite{BBexciting} fail in the absence of an explicit diffusion coefficient for $M$.

\subsection{Connections to the literature}\label{sec:lite}

In the case of $d=1$, the optimization problem \eqref{eq:optimal} was solved in \cite{BBexciting} by a variational argument.  The first order condition yields an ODE for the optimal diffusion coefficient, which has an explicit solution. The authors then verified that the unique optimal win-martingale is characterized by the SDE
$$\textstyle dM_s=\frac{\sin(\pi M_s)}{\pi\sqrt{1-s}} \,dB_s.$$ 
Obviously we recover this result from \eqref{eq:Monge} and \eqref{eq:AMTG}, with $g(x)=\log(\pi^2/\sin^2(\pi x))$.

Another article closely related to ours is \cite{GuHoPoRe23} by Guo, Howison, Possama\"i and Reisinger, who also aim to address Aldous' question \cite{Al22b,Al22a} on the most random win-martingale / most exciting game via a PDE perspective. 
In contrast to us, that work is concerned with the one-dimensional case only. Furthermore, the win-martingales considered there are allowed to stop strictly before the terminal time. Mathematically this translates into introducing the first hitting time $\tau$ where the martingale leaves $\Delta_1:=(0,1)$, and yields a zero boundary condition for its HJB equation. Hence in that work the diffusion coefficient of the optimal win-martingale (and the value function) does not separate neatly into time- and space-dependent parts. As a result, in the context of \cite{GuHoPoRe23} it is not possible to reduce the HJB equation into an elliptic Monge-Amp\`ere equation.

As observed in \cite{BBexciting}, the objective function in \eqref{eq:optimal} is related to the so-called \emph{specific relative entropy}, which was introduced by Gantert \cite{Ga91} as the scaling limit of relative entropies. In \cite{BaBe24,BaUn23}, under given assumptions, the authors proved that the specific relative entropy of a $d$-dimensional martingale measure $\mathbb Q$ w.r.t.\ the Wiener measure  is given by 
\begin{align*}\textstyle
   \frac{1}{2} \mathbb E_{\mathbb Q} \left[\int_0^1 \{\Tr(\Sigma_s)-\log\det(\Sigma_s)-d \} \,ds \right].
\end{align*}
Since the expected integral of $\Tr(\Sigma_s) -d$ is independent of the choice of $\mathbb Q \in \mathcal{M}_{t,x}^{win}$, the optimization problem \eqref{eq:optimal} is akin to finding the $d$-dimensional win-martingale that minimizes the specific relative entropy w.r.t.\ Wiener measure. In dimension $d=1$ it was argued in \cite{BBexciting} that this specific relative entropy minimization problem over continuous win-martingales is associated to a scaling limit of Shannon entropy maximization problems over discrete-time win-martingales: This is the connection to Aldous original question. In recent times, the specific relative entropy has gained more visibility through the works of F\"ollmer \cite{Fo22b,Fo22a}, wherein e.g.\ a novel transport-information inequality has been obtained. Over the years, the concept of specific relative entropy between martingales has been rediscovered independently, e.g.\ in \cite{CoDo22} by Cohen and Dolinsky or in  \cite{Av01} by  Avellaneda, Friedman, Holmes and Samperi. The function  on covariance matrices $\Tr(\cdot)-\log\det(\cdot)-d$ has a long history in statistics (aka Stein loss  \cite{DeSr85,JaSt92}) and more recently machine learning (aka Log-Det divergence  \cite{Da07,KuMaDh06}).

Our problem \eqref{eq:optimal} belongs to the class of martingale optimal transport problems. Following \cite{BeHePe12, BoNu13,GaHeTo13,HoNe12},
martingale versions of the classical transport problem (see e.g.\ \cite{FiGl21,Sa15,Vi03, Vi09} for recent monographs) are often considered due to applications in  mathematical finance but admit further applications, e.g.\ to the Skorokhod problem \cite{BeCoHu14, BeNuSt19}.  In analogy to classical optimal transport, necessary and sufficient conditions for optimality have been established for martingale transport problems in discrete time (\cite{BeJu21, BeNuTo16}). However in continuous time there are very few instances where these problems can be solved in semi-explicit form; see e.g.\ \cite{BaBeHuKa20,BaLoOb24,BePaRi24,Lo18} for some exceptions. For a dual perspective to martingale optimal transport in continuous time see Huesmann and Trevisan's \cite{HuTr19}, or more generally Tan and Touzi's \cite{TaTo13}. These dual problems are optimization problems that range over a family of HJB equations, while in our case we have a single HJB equation. This is because the constraint of terminating at the vertices has been resolved by discovering the ``correct'' terminal condition of the value function, and thus no Lagrangian multiplier is needed. 

The Monge-Amp\`{e}re equation appears in various well-known problems, including the Weyl problem, the Minkowski problem, the existence of K\"{a}hler-Einstein metrics, the optimal transport problem; see e.g.\ \cite{Ca92,MC97,CNS1, Fi17, Le24}.  In \cite{CNS1}, Caffarelli, Nirenberg, and Spruck  investigated its Dirichlet problem in strictly convex domains of $\mathbb{R}^n$. For the infinite  boundary value problem as \eqref{eq:Monge}, Cheng-Yau \cite{Cheng-Yau1,Cheng-Yau2} initially analyzed this equation arising from complex geometry. Later, Guan and Jian \cite{GuJi04} established the existence and uniqueness of the  Monge-Amp\`{e}re equation with infinite boundary condition defined on general convex domain.  For further discussions on the infinite boundary value problem, see \cite{Trombetti,Jian,MoAh,XiangYang} and the  references therein. Additionally, real Monge-Amp\`{e}re equations on polytopes have been explored in complex geometry; see e.g.\ \cite{ ChenLiSheng2,ChenLiSheng1,ZhuWang}.

\subsection{Structure of the paper} In Section~\ref{sec2}, we briefly  describe a connection between Aldous martingales, moment measures, and a class of Langevin dynamics, which we think is of independent interest. The reader can skip it if they are only concerned with Theorems \ref{thm:quantitive}-\ref{thm:intro_existence}. In Section~\ref{sec3}, we collect some properties of the value function defined in \eqref{eq:optimal}, namely, its finiteness in Proposition~\ref{prop:finite}, the terminal-boundary condition in Lemma~\ref{lem:boundary-terminal}, and the convexity in Proposition~\ref{prop:convex}. In Section~\ref{sec4}, we derive the HJB equation and show that the value function \eqref{eq:optimal} is the unique viscosity solution to its corresponding HJB equation. In Section~\ref{sec5} and Section~\ref{sec6} we prove Theorem~\ref{thm:quantitive} and Theorem~\ref{thm:intro_existence} respectively.

\subsection{Notation}
\begin{itemize}
\item $\Delta\subset\mathbb R^d$ is the subprobability simplex in $d$-dimensions;

    \item $\{e_1,\dots,e_d\}$ is the canonical basis of $\mathbb R^d$, and $e_0:=0\in\mathbb R^d$;
    \item $\mathcal{V}=\{e_0,\dotso,e_d\}$ denotes the vertices of $\Delta$;
    \item $\bf 1$ is the vector of ones of length $d$;
    \item $\mathbb S^d$ denotes the set of $d \times d$ positive semi-definite matrices;
    \item $1_{d \times d}$ is the $d$-dimensional identity matrix;
    \item $\top$ denotes transpose;
    \item $\lambda_{min}$, $\Tr$, $\det$ denote minimal eigenvalue, trace, and determinant of matrices;
    
    \item $\mathcal{S}^{d-1}$ denotes the unit sphere in Euclidean space $\mathbb R^d$;
    \item $\#$ denotes the push-forward operator between a function and a measure.
\end{itemize}

\section{Connection to ``moment measures''}\label{sec2}

We adopt a variant of the notion of moment measure: For a convex function $\phi$, its ``moment measure'' is defined\footnote{For the necessary background on moment measures, see \cite{CEKl15,Fa19,Sa16} and the references therein. Usually the moment measure of a convex function $\phi$ is a finite measure $\mu^\phi$ satisfying $\mu^\phi=(\nabla \phi)_{\#}(e^{-\phi(x)}dx) $. In our case we do not have the minus sign in the exponent, and the measure is only sigma-finite. At a lack of a better name we will call these objects ``moment measures'', i.e.\  employ quotation marks.} to be 
\[\mu^{\phi}:=(\nabla \phi)_{\#} \left(e^{\phi(x)}\mathbbm{1}_{\{\phi(x) < +\infty\}}\,dx\right). \] 
For $\phi:= g$, and owing to the fact that $\nabla g$ is a diffeomorphism (see Lemma~\ref{lem:diffeomorphism}), it can be seen that the Monge-Amp\`{e}re Equation \eqref{eq:Monge} is equivalent to the condition $$\mu^{g}=2^{-d}\text{Leb}_{\mathbb{R}^d}.$$ For instance, for $d=1$, we have $g(x)=\log(\pi^2/\sin^2(\pi x))$. 
To simplify computations we pass now from the Aldous martingale $M$ defined in \eqref{eq:AMTG} to its time-changed version $Y_t:=M_{1-e^{-t}}$, so that now $t\in \mathbb R_+$. It can be verified that  
    \begin{align}\label{eq:Y_intro}\textstyle 
        dY_t= \sqrt{2(\nabla^2g(Y_t))^{-1}}\,dW_t, 
    \end{align}  
    for a $d$-Brownian motion $(W_t)_{t \in\mathbb R_+}$. By It\^{o}'s formula and taking gradient in \eqref{eq:Monge}, we have 
\begin{align*}
d \nabla g(Y_t)=& \textstyle  \sqrt{2 \nabla^2 g(Y_t) } \, dW_t+ \nabla^3 g(Y_t) \nabla^2 g(Y_t)^{-1} dt 
=\sqrt{2  \nabla^2 g^*( \nabla g (Y_t))^{-1}} \, dW_t + \nabla g(Y_t) \,dt, 
\end{align*}
with $g^*$ standing for the convex conjugate of $g$.
Hence if we define $X_t:=\nabla g(Y_t)$ we find 
\begin{align}\label{eq:generalized_repulsive_OU}\textstyle 
dX_t=X_t \,dt + \sqrt{ 2 (\nabla^2 g^*(X_t))^{-1}} \, d W_t.
\end{align}
An appealing property\footnote{We thank M.\ Fathi and S.\ Chewi, who explained to an author the connection between moment measures (i.e.\ without quotation marks) and processes akin to \eqref{eq:generalized_repulsive_OU} but with a minus sign in the drift. Such processes go under the name of Newton-Langevin diffusion in the literature and appear in sampling problems; see e.g.\ \cite{ChLGLuMaRiSt20,ZhPeFaPe20}.} of a ``moment measure'', such as $\mu^{g}=\frac{1}{2^d}\text{Leb}(\mathbb{R}^d)$, is that it is a sigma-finite invariant measure for the diffusion $X$. 
Indeed, it is sufficient to check that for any suitable test function $f$ 
\begin{align*}\textstyle 
\int \nabla f(x)^\top  x  \, d \mu^{g}= -\int \Tr[ \nabla^2 g^*(x)^{-1} \nabla^2 f(x)] \, d \mu^g. 
\end{align*}
By change of measure and integration by parts, this is indeed the case: 
\begin{align*}\textstyle 
\int \nabla f(x)^\top  x  \, d \mu^{g}
&=\textstyle \int \nabla f(\nabla g(x))^\top \nabla e^{g(x)} \,dx =- \int \Tr[   \nabla^2 f(\nabla g(x)) \nabla^2 g(x)]e^{g(x)} \, dx, \\
 &=\textstyle  - \int \Tr[ \nabla^2f(\nabla g(x)) \nabla^2 g^*(\nabla g(x))^{-1} ] e^{g(x)} \, dx \\
 &=\textstyle -\int \Tr[ \nabla^2 g^*(x)^{-1} \nabla^2 f(x)] \, d \mu^g, 
\end{align*}
where in the third equality we use that $\nabla^2 g(x) = (\nabla^2 g^* \circ \nabla g(x))^{-1}$.  As a consequence of the definition of $\mu^g$, we have that $$(\nabla g^{-1})_\#(\mu^{g})=  e^{g(x)}\mathbbm{1}_{\{g(x) < +\infty\}} \, dx,$$ and the r.h.s.\ is a sigma-finite invariant measure for $(Y_t)=(\nabla g^{-1}(X_t))$. In summary:

\begin{itemize}
\item  The martingale $Y$, given in \eqref{eq:Y_intro}, and the diffusion with drift $X$, given in \eqref{eq:generalized_repulsive_OU}, can be coupled so that  $\nabla g(Y_t)=X_t$;

\item $e^{g(x)}\mathbbm{1}_{\{g(x) < +\infty\}} \, dx$ is an invariant sigma-finite measure of $Y$;
\item $\mu^g=2^{-d}\text{Leb}_{\mathbb{R}^d}$ is an invariant sigma-finite measure of $X$.
\end{itemize}

In these arguments the simplex $\Delta$, equal to $\{g(x) < +\infty\}$, played no role. Dropping it, one would work on the domain $\{g(x) < +\infty\}$, provided  $\nabla g$ is a diffeomorphism on this set.

\section{On the value function }\label{sec3}

We begin this part with the precise definition of our value function: We define for $t\in[0,1)$ and $x\in \Delta$ the quantity $v(t,x)$ as the infimum of 
\[\textstyle \mathbb{E}_{{\mathbb P}} \left[\int_t^1  -\log \det(\Sigma_s) \, ds  \right],\]
taken over the class of all filtered probability spaces $({\Omega},{\mathcal F }, ({\mathcal F }_u)_{t \leq u\leq 1},{\mathbb P})$ supporting a $d$-dimensional Brownian motion $ B$ and a martingale $(M_u)_{t\leq u\leq 1}$  with $dM_u=\sigma_ud  B_u$,  $M_t=x$ and $M_1\in \mathcal{V}$ a.s., and where $ \Sigma_s:= \sigma_s\sigma_s^\top$. The flexibility of choosing the probability space will be useful for some arguments. On the other hand, it is not difficult\footnote{Indeed, we can take $\mathbb Q=\text{Law}(M)$. Then $\langle M\rangle_u = \int_0^u \bar \Sigma_s\circ M \,ds $, where for instance $\bar \Sigma_s:=\liminf_{n\to\infty} n\{\langle X \rangle_s - \langle X \rangle_{(s-1/n)_+}  \}$, with $X$ denoting the canonical process. Thus $\bar \Sigma\circ M$ is a version of $\Sigma$ (in the $d\mathbb P\times ds$-sense) and $\mathbb{E}_{{\mathbb P}} \left[\int_t^1  -\log \det(\Sigma_s) \, ds  \right] = \mathbb{E}_{\text{Law}(M)} \left[\int_t^1  -\log \det(\bar\Sigma_s) \, ds  \right]$. } to see that $v(t,x)$ is also equal to the infimum of $\mathbb{E}_{{\mathbb Q}} \left[\int_t^1  -\log \det(\Sigma_s) \, ds  \right]$ where now $\mathbb Q$ is a probability measure on the path-space of $\mathbb R^d$-valued continuous functions on $[t,1]$ such that the canonical process is a continuous martingale started at $x$ at time $t$ and taking values in $\mathcal{V}$ at time $1$. In this case $\Sigma$ denotes (a version of) the density of the quadratic covariation of the canonical process under $\mathbb Q$. Henceforth we can and will switch between the two formulations of the problem, and we will denote in either case by $\mathcal{M}^{win}_{t,x}$ the class over which we are minimizing.

We continue by showing elementary properties of the value function:

\begin{proposition}\label{prop:finite}
    For $(t,x)\in [0,1) \times \Delta$, we have $-\infty <v(t,x)<\infty$.
\end{proposition}

\begin{proof}
    We prove this by induction on the dimension $d$, so it will be convenient to write $\Delta_d\subset \mathbb R^d$ rather than just $\Delta$ and $v^d$ rather than just $v$. The bound $-\infty <v^d(t,x)$ is trivial. Indeed according to the strict convexity of $\mathbb S^d \ni \Sigma \mapsto \Tr(\Sigma)-\log\det \Sigma -d$, a variational argument shows that the unique minimizer of this function is the identity matrix $I_{d \times d}$, so 
    \begin{align}\label{eq:logdetbound}
    -\log\det \Sigma \geq -\Tr \Sigma+d, \quad \, \forall \, \Sigma \in \mathbb S^d.
    \end{align}
    The r.h.s.\ has a finite constant value when integrated over win-martingales started at $x$, and hence $v^d(t,x)$ is bounded from below. In the following we show $v^d(t,x)<\infty$. W.l.o.g.\ we can assume that $t=0$.

    For $d=1$ we know this already by \cite{BBexciting}. However, for the inductive step, it will be helpful to exhibit a feasible one-dimensional win-martingale with certain pleasant properties. Given $x\in\Delta_1$ we let the $1$-dimensional martingale $M$ to be determined via $M_0=x$ and $dM_t=\sqrt{\frac{2}{1-t}}M_t(1-M_t)dB_t $. According to \cite{BaZh24}, $M$ is a win-martingale, i.e., $M \in \mathcal{M}_{0,x}^{win}$. We now show the finiteness of the following expectations:
    \begin{align}\label{eq:example_fav_finiteness}
     \textstyle    \mathbb E\left[\int_0^1-\log(M_t) dt \right],\, \mathbb E\left[\int_0^1 -\log(1-M_t) dt \right],\, \mathbb E\left[\int_0^1-\log\left(\sigma^2_t\right)dt,  \right] 
    \end{align}
    with $\sigma_t^2:= \frac{2}{1-t}M_t^2(1-M_t)^2$. Clearly it suffices to check the finiteness of the first of these. By It\^{o}'s formula we have
    \[\textstyle -d \log M_t =-\sqrt{\frac{2}{1-t}}(1-M_t)dB_t +\frac{1}{1-t}(1-M_t)^2dt .\]
    We remark that the local martingale part is a true martingale for $t\in [0,1-\epsilon]$ and $0<\epsilon<1$ arbitrary, since the integrand is then bounded. Hence
    \begin{align*}
    (1-\epsilon)\log(x) &\textstyle +    \mathbb E\left[\int_0^{1-\epsilon}-\log(M_t) dt \right] =   \mathbb E\left[\int_0^{1-\epsilon}\int_0^t  \frac{(1-M_s)^2}{1-s}dsdt \right] \\
    &= \textstyle  \mathbb E\left[\int_0^{1-\epsilon}\int_s^{1-\epsilon}  \frac{(1-M_s)^2}{1-s}dtds \right] \leq \mathbb E\left[\int_0^{1}\int_s^{1}  \frac{(1-M_s)^2}{1-s}dtds \right]
    \\&\textstyle = \mathbb E\left[\int_0^{1}(1-M_s)^2ds \right] \leq 1.
    \end{align*}
    To conclude we send $\epsilon \to 0$, and apply monotone convergence by observing that $-\log(M_t)$ is lower bounded by the integrable process $-M_t+1$. 

    We now assume that we have proved $v^d(0,x)<\infty$ for every $x\in\Delta_d$, and we want to show that for every $y\in\Delta_{d+1}$ there is a win-martingale started at $y$ with instantaneous covariance $\bar\Sigma$ such that $\mathbb E[\int_0^1 -\log\det(\bar\Sigma_t)dt]<\infty$. We write $x_i:=y_i/(\sum_{k\leq d+1}y_k)$ for $i\leq d$, and observe that $x\in\Delta_d$. We let $X$ be a win-martingale starting from $x$, with instantaneous covariance $\Sigma_t$, such that $\mathbb E[\int_0^1 -\log\det(\Sigma_t)dt]<\infty$, whose existence is guaranteed by the inductive hypothesis. We also let $M$ be a win-martingale starting from $\sum_{k\leq d+1}y_k$, with volatility $\sigma^2$ satisfying $-\infty <\mathbb E[\int_0^1 -\log(\sigma^2_t)dt]<\infty$ and $-\infty< \mathbb E[\int_0^1- \log(M_t)dt]<\infty$; see \eqref{eq:example_fav_finiteness} for the sake of concreteness. We couple $M$ and $X$ independently and define 
    \[Y_t:= M_t(X^1_t,\dots,X^d_t,1-\mathbf 1^{\top} X_t)^\top,\]
    where we denote by $\top$ the transpose and by $\mathbf 1$ the column vector of ones with length $d$. One readily checks that $Y_0=y$ and that $Y_1$ takes values on the vertices of $\Delta_{d+1}$. By independence of $X$ and $M$, we see that $Y$ is also a martingale. Thus it is a win-martingale started at $y$. Denoting its instantaneous covariance by $\bar\Sigma$, one finds
    \begin{align*}
        \bar\Sigma_t^{i,j}&=M_t^2\Sigma_t^{i,j} + X^i_tX^j_t\sigma^2_t, \, \text{for }i,j\leq d;
\\         \bar\Sigma_t^{i,d+1}&=-M_t^2 \mathbf 1^{\top}\Sigma_t^{\cdot,i} + X^i_t(1-\mathbf 1^{\top}X_t)\sigma^2_t, \, \text{for }i\leq d;\\
\bar \Sigma_t^{d+1,d+1} &= M_t^2 \mathbf 1^{\top}\Sigma \mathbf 1 + (1-\mathbf 1^{\top}X_t)^2\sigma^2_t.
    \end{align*}
    Let as call $v_t^{\top}:=-M_t^2 \mathbf 1^{\top}\Sigma_t + X_t^{\top}(1-\mathbf 1^{\top}X_t)\sigma^2_t $ and $A_t:=M_t^2\Sigma_t + X_tX_t^{\top}\sigma_t^2 $, so that 
    \[v_t^{\top}= X_t^{\top}\sigma_t^2 - \mathbf 1^{\top} A_t,\]
    as well as 
    \[\bar\Sigma_t^{d+1,d+1} = \mathbf 1^{\top} A_t \mathbf 1 + \sigma_t^2- 2\sigma_t^2 \mathbf 1^{\top} X_t. \]
    With this notation,     $\bar\Sigma_t$ admits the block form
    $$ \bar\Sigma_t =   \left[ 
    \begin{array}{c|c} 
      A_t & v_t \\ 
      \hline 
      v_t^{\top} &  \bar \Sigma_t^{d+1,d+1}
    \end{array} 
    \right] = 
    \left[ 
    \begin{array}{c|c} 
      A_t & X_t\sigma_t^2 -  A_t\mathbf 1 \\ 
      \hline 
      X_t^{\top}\sigma_t^2 - \mathbf 1^{\top} A_t &  \mathbf 1^{\top} A_t \mathbf 1 + \sigma_t^2- 2\sigma_t^2 \mathbf 1^{\top} X_t
    \end{array} 
    \right] 
    . $$
    As $A_t$ is necessarily invertible, by properties of the determinant we get
    \[\det \bar\Sigma_t = \det A_t \cdot \{\bar\Sigma_t^{d+1,d+1} - v_t^{\top} A_t^{-1}  v_t\} =  \det A_t \cdot\{\sigma_t^2-\sigma_t^4 X_t^{\top}A_t^{-1}X_t\}  .\]
    Since $X_tX_t^{\top}$ is positive semidefinite, by properties of the determinant we get 
    \[-\log\det A_t \leq -\log\det(M_t^2\Sigma_t) = -2d\log(M_t) - \log\det \Sigma_t, \]
    whose r.h.s.\ is integrable. Thus everything boils down to showing
    \begin{align}
    \label{eq:interim_goal}\textstyle 
    \mathbb E\left[ \int_0^1 - \log \{1-\sigma_t^2 X_t^{\top}A_t^{-1}X_t\}dt\right ] <\infty,
    \end{align}
for which we have used the fact that $\log \sigma_t^2$ is integrable. 

According to Sylvester's determinant theorem, 
\begin{align*}
    1-\sigma_t^2 X_t^{\top}A_t^{-1}X_t = \det \left(I_{d\times d} - \sigma_t^2 X_t X_t^{\top}A_t^{-1} \right)=\det (A_t - \sigma_t^2 X_t X_t^{\top} )/\det A_t.
\end{align*}
Due to the definition of $A_t$, $A_t - \sigma_t^2 X_t X_t^{\top}=M_t^2\Sigma_t$, and by the independence of $M$ and $X$, $\mathbb E \left[\int_0^1 -\log\det(M_t^2\Sigma_t) \, dt \right]<+\infty$. Therefore \eqref{eq:interim_goal} is equivalent to that 
$\mathbb E \left[ \int_0^1 \log \det A_t \, dt \right]< \infty.$
Noting that $M_t^2 \leq 1$ and $\Tr(X_tX_t^{\top})\leq 1$, and  invoking \eqref{eq:logdetbound}, we conclude that 
\begin{align*}\textstyle 
    \mathbb E \left[ \int_0^1 \log \det A_t \, dt \right] \leq \mathbb E \left[ \int_0^1  \left( \Tr A_t -d \right)  dt \right] \leq \mathbb E \left[ \int_0^1 \left( \Tr \Sigma_t + \sigma_t^2 \right) dt \right]-d,
\end{align*}
where the last term is finite since $M$ and $X$ are bounded martingales.

\end{proof}

 The next result shows that $v(t,x)=(1-t)v(0,x)+d (1-t) \log(1-t)$. As mentioned in the introduction, $g(x):=v(0,x)$ satisfies the Monge-Amp\`{e}re equation \eqref{eq:Monge} and will be studied in detail in later sections.

\begin{lemma}\label{lem:boundary-terminal}

For any $t,s \in [0,1)$, $x \in \Delta$, we have the equality
    \begin{align*}
        \frac{1}{1-s}v(s,x)-d \log (1-s)= \frac{1}{1-t} \,  v(t,x)-d \log(1-t).
    \end{align*}
Moreover, for any $(t^0,x^0) \in [0,1) \times \pa \Delta $, 
\[ \lim\limits_{(t,x)\to (t^0,x^0)} v(t,x)= +\infty.\]
\end{lemma}
\begin{proof} 
    For any $\epsilon>0$, suppose $(M^t_u)_{u \in [t,1]}$ with distribution $\mathbb Q^t \in \mathcal{M}_{t,x}^{win}$ is a win-martingale, whose density of quadratic variation we denote $\Sigma^t$, such that 
\begin{align*}\textstyle 
    v(t,x) \leq \E_{\mathbb Q^t}\left[\int_t^1 - \log \det \Sigma^t_u \, du \right] <v(t,x)+\epsilon. 
\end{align*}
Let us take $M^s_u= M^t_{t+(u-s)(1-t)/(1-s)}$ for $u \in [s,1]$, and denote its distribution by $\mathbb Q^s$ and the density of the quadratic variation by $\Sigma^s$. Therefore
\begin{align*}
v(s,x) \leq & \textstyle \E_{\mathbb Q^s} \left[\int_s^1 -\log \det \Sigma^s_u \,du \right]= \E_{\mathbb Q^s} \left[\int_s^1 -\log \det\left (\frac{1-t}{1-s} \Sigma^t_{t+(u-s)(1-t)/(1-s)} \right ) \,du \right]\\
=&\textstyle \frac{1-s}{1-t} \, \E_{\mathbb Q^t}\left[- \int_t^1 \log\det \left( \frac{1-t}{1-s} \Sigma^t_{u} \right) \,du \right] \\
=& \textstyle d(1-s) \log\left(\frac{1-s}{1-t}\right)+\frac{1-s}{1-t} \, \E_{\mathbb Q^t}\left[-\int_t^1 \log \det \Sigma^t_u \, du \right] \\
\leq &\textstyle  d(1-s) \log\left(\frac{1-s}{1-t}\right)+ \frac{1-s}{1-t} \, ( v(t,x)+\epsilon).
\end{align*}
Letting $\epsilon \to 0$, we get half of the equality. Exchanging the roles of $t$ and $s$, we finish proving the first claim.

    For any $\mathbb Q \in \mathcal{M}_{t,x}^{win}$, the martingale property yields that 
   \begin{align*}\textstyle 
       \E_{\mathbb Q}\left[\int_t^1  \Sigma_s \,ds  \right]= Diag(x_1,\dotso,x_d)- xx^\top ,
   \end{align*}
    where $Diag(x_1,\dotso,x_d)$ denotes the $d$-dimensional diagonal matrix with the $(i,i)$-entry being $x_i$ for $i=1,\dotso, d$. Recall that $\mathbb S^d \ni \Sigma \mapsto \log \det \Sigma$ is concave. We get a lower bound, for any $\mathbb Q \in \mathcal{M}^{win}_{t,x}$
    \begin{align*}\textstyle 
        \E_{\mathbb Q} \left[\int_t^1 -\log \det \Sigma_s \,ds \right] \geq &\textstyle  -(1-t)\log \det \E_{\mathbb Q} \left[\frac{1}{1-t}\int_t^1 \Sigma_s \,ds \right] \\
        =&\textstyle  d(1-t)\log(1-t)-(1-t) \log \det \E_{\mathbb Q}\left[\int_t^1  \Sigma_s \,ds  \right] \\
        \geq &\textstyle  d(1-t)\log(1-t)-(1-t)\log \det \left(Diag(x_1,\dotso,x_d)-x x^\top  \right).
    \end{align*}
    Noticing that $$\lim_{x \to x^0 \in \pa\Delta}\log \det \left(Diag(x_1,\dotso,x_d)-x x^\top  \right)=-\infty,$$ we would finish the proof of the second claim. To this aim, observe that if $x^0$ has a null coordinate, say coordinate $i$, then for the matrix $Diag(x^0_1,\dotso,x^0_d)-x^0( x^0)^\top$ its $i$-th row and column are entirely null, and therefore this matrix is non-invertible. If on the other hand $x^0$ has no null coordinates, then $\sum_i x^0_i=1$, and so by the Sherman-Morrison formula we have
    \begin{align*}
    &\det \left(Diag(x^0_1,\dotso,x^0_d)-x^0 (x^0)^\top  \right)  \\
    &= \det Diag(x^0_1,\dotso,x^0_d)  \times [1-(x^0)^\top Diag(1/x^0_1,\dotso,1/x^0_d) x^0 ],
    \end{align*}
    and the latter factor is zero.
\end{proof}

The following result will be used in the next section to show that the value function is the unique \emph{convex} viscosity solution to its HJB equation.

\begin{proposition}\label{prop:convex}
    For $t\in [0,1)$ fixed, the function $\Delta \ni x\mapsto v(t,x)$ is a finite convex function.
\end{proposition}

\begin{proof}

    W.l.o.g.\ we can assume $t=0$. Now fix $\bar x,\hat x\in \Delta$, $\lambda\in[0,1]$, and take $\epsilon>0 $ arbitrary. As $v(0,\bar x),v(0,\hat x)<\infty$, we can find martingales laws $\bar{\mathbb Q}$ resp.\ $\hat{\mathbb Q}$, started at $\bar x$ resp.\ $\hat x$ that are $\epsilon$-optimizers for $v(0,\bar x)$ resp.\ $v(0,\hat x)$. Denoting by $\bar \Sigma$ resp.\ $\hat \Sigma$ their associated instantaneous covariances, we conclude that these  are a.s.\ non-singular. By Theorem 4.2 and Remark 4.3 of \cite[Chapter 3]{KaSh91}, we can realize $\bar{\mathbb Q}$ as the law of a martingale $\bar M$ satisfying $d\bar M_t=\bar \sigma_t d\bar B_t$ for some $d$-dimensional Brownian motion, with $\bar \sigma_t $ the symmetric positive square root of $\bar\Sigma_t$. Similarly, we can realize $\hat{\mathbb Q}$ as the law of a martingale $\hat M$ satisfying $d\hat M_t=\hat\sigma_t d\hat B_t$ for some $d$-dimensional Brownian motion, with $\hat \sigma_t $ the symmetric positive square root of $\hat\Sigma_t$. We can then couple all of these processes by setting $\bar B=\hat B=:B$.  By possibly enlarging the probability space, we can take $(p_t)_t$ to be any one-dimensional win-martingale, with $p_0=\lambda$, independent of $(\bar M,\hat M, B)$. Thus $p$ is driven by a likewise independent one-dimensional Brownian motion $B^{d+1}$, and $dp_t=\eta_tdB^{d+1}_t$ for some $(\eta_t)_t$. We now define 
    \[M_t:=p_t\bar M_t + (1-p_t)\hat M_t.\]
    Then $M_0=\lambda \bar x + (1-\lambda)\hat x$ and crucially $M_1$ takes values in the vertices of $\Delta$. Moreover $M$ is a martingale. To wit, by independence 
    \[\textstyle dM^i_t = \sum_{j=1}^d \{p_t\bar \sigma_t^{ij} + (1-p_t)\hat\sigma_t^{ij} \}dB^j_t + \{\bar M^i_t - \hat M^i_t\}\eta_t dB^{d+1}_t.\]
Thus $M$ is a win-martingale started at $\lambda \bar x + (1-\lambda)\hat x$. We compute
\[\textstyle d\langle M^i,M^k  \rangle_t = \sum_{j=1}^d \{p_t\bar \sigma_t^{ij} + (1-p_t)\hat\sigma_t^{ij} \}\{p_t\bar \sigma_t^{kj} + (1-p_t)\hat\sigma_t^{kj} \} dt +\eta_t^2\{\bar M^i_t - \hat M^i_t\}\{\bar M^k_t - \hat M^k_t\}  dt, \]
so the instantaneous covariance matrix of $M$  is
\[\Sigma_t:= (p_t\bar\sigma_t+(1-p_t)\hat\sigma_t)^2+\eta_t^2(\bar M_t- \hat M_t)(\bar M_t- \hat M_t)^{\top}.\]
By Sylvester's determinant theorem we have
\begin{align*}
  & \log \det \Sigma_t \\ = &\log \det [(p_t\bar\sigma_t+(1-p_t)\hat\sigma_t)^2 ] + \log \det [I_{d\times d}+\eta_t^2(p_t\bar\sigma_t+(1-p_t)\hat\sigma_t)^{-2}(\bar M_t- \hat M_t)(\bar M_t- \hat M_t)^{\top}] \\
  =& \log \det [(p_t\bar\sigma_t+(1-p_t)\hat\sigma_t)^2 ] + \log  [1 + \eta_t^2(\bar M_t- \hat M_t)^{\top}(p_t\bar\sigma_t+(1-p_t)\hat\sigma_t)^{-2}(\bar M_t- \hat M_t)].
\end{align*}
Let us call $r_t:= \log  [1 + \eta_t^2(\bar M_t- \hat M_t)^{\top}(p_t\bar\sigma_t+(1-p_t)\hat\sigma_t)^{-2}(\bar M_t- \hat M_t)]$.
Thus
\[\log\det \Sigma_t = \log\det\left ( (p_t\bar\sigma_t+(1-p_t)\hat\sigma_t)^2 \right )+r_t\geq p_t\log\det\bar \Sigma_t +(1-p_t)\log\det \hat\Sigma_t+r_t, \]
as follows from the concavity of the log-determinant.
Integrating, using the independence of $p$ and $(\bar\Sigma,\hat\Sigma)$, and  recalling that $\mathbb{E}[p_t]=\lambda$, we obtain
\begin{align*}
    v(0,\lambda\bar x + (1-\lambda)\hat x) &\textstyle \leq \mathbb{E}\left[\int_0^1-\log\det\Sigma_t \, dt\right]  \\ &\textstyle \leq  \mathbb{E}\left[\int_0^1 -p_t\log\det\bar \Sigma_t \, dt\right] +\mathbb{E}\left[\int_0^1-(1-p_t)\log\det \hat\Sigma_t \, dt\right]  -\mathbb{E}\left[\int_0^1r_t \, dt\right] \\ &\textstyle = \lambda \mathbb{E}\left[\int_0^1 -\log\det\bar \Sigma_t \, dt\right] + (1-\lambda )\mathbb{E}\left[\int_0^1-\log\det \hat\Sigma_t \, dt\right]-\mathbb{E}\left[\int_0^1r_t \,dt\right]  \\
    &\textstyle  \leq  \lambda v(0,\bar x) + (1-\lambda) v(0,\hat x)+\epsilon -\mathbb{E}\left[\int_0^1r_tdt\right] .
\end{align*}
We can now send $\epsilon$ to zero. As $r_t\geq 0$,  we thus obtain the convexity inequality for $v(0,\cdot)$. 
\end{proof}

\section{Viscosity solutions point of view}\label{sec4}

In the interior $[0,1) \times \Delta$, the HJB equation associated to \eqref{eq:optimal} is  clearly $$\textstyle 
-\pa_t v(t,x)= \inf_{\Sigma \geq 0}\left\{-\log\det(\Sigma)+\frac{1}{2} \Tr(\nabla^2 v \Sigma ) \right\}.$$
Due to the convexity of $\Sigma \mapsto -\log\det(\Sigma)+\frac{1}{2} \Tr(\nabla^2 v \, \Sigma )$, there exists a unique minimizer $\Sigma_*$ satisfying 
$
0=\Tr\left(\frac{1}{2}\nabla^2 v \Delta \Sigma \right) - \Tr(\Sigma_*^{-1} \Delta \Sigma)
$ 
for any variation $\Delta \Sigma=\Sigma'-\Sigma_*$. Therefore, heuristically,  $\Sigma_*^{-1}=\frac{1}{2}\nabla^2 v$, and the right hand side of the equation above becomes $$\textstyle \log \det \left(\frac{1}{2}\nabla^2 v\right)+d .$$ 
From Lemma~\ref{lem:boundary-terminal}, $v$ is infinite on the boundary and zero at the terminal time, and therefore we obtain \eqref{eq:HJB}. In this section, we show that the value function $v$ is the unique \emph{viscosity solution} to \eqref{eq:HJB} over \emph{convex functions}.

{ We say that a function $u: [0,1) \times \Delta \to \R$ is convex in space if for any $t \in [0,1)$, $x \mapsto u(t,x)$ is convex. Let us consider equations of the form
\begin{equation}\label{eq:gHJB}
\begin{cases}
 -\pa_t v(t,x)=H(t,v(t,x), \det \nabla^2 v(t,x)), \hspace{60pt} (t,x) \in [0,1) \times \Delta,     \\
\quad  \, \ v(t,x)=+\infty, \quad  \hspace{132pt} (t,x) \in [0,1) \times \partial \Delta, \\
\quad  \  v(1,x)=0, \quad \quad \hspace{130pt} (1,x) \in \{1\} \times \Delta,
\end{cases}
\end{equation}
where $H: [0,1) \times \R \times \R_+ \to \R$ is a continuous Hamiltonian such that  $H(t,u,\cdot)$ is increasing for any $t \in [0,1)$, $u \in \R$.  Setting $H(t,u,a)=\log a-d\log 2 +d$, \eqref{eq:gHJB} recovers \eqref{eq:HJB}. We introduce the notion of viscosity solution to general equations of the form \eqref{eq:gHJB} for the proof of Proposition~\ref{prop:uniqueness}.  }

\begin{definition}
A convex function $u_1: [0,1) \times \Delta \to \R$ is said to be a viscosity sub-solution to \eqref{eq:gHJB} if 
\begin{enumerate}
    \item[(i)] For any $x_0 \in \Delta$, $\limsup_{(t,x) \to (1,x_0)} u_1(t,x) \leq 0$. 
    \item[(ii)] For any convex test function $\phi \in C^2([0,1) \times \Delta; \R)$, and any local maximum $(t_0,x_0) $ of $u_1 - \phi$, we have 
    \begin{align*}
        -\pa_t \phi(t_0,x_0) \leq H(t_0, u_1(t_0,x_0), \det \nabla^2 \phi(t_0,x_0)).
    \end{align*}
\end{enumerate}
A convex function $u_2: [0,1) \times \Delta \to \R$ is said to be a viscosity super-solution to \eqref{eq:gHJB} if 
\begin{enumerate}
    \item[(i)] For any $x_0 \in \bar{\Delta}$, $\liminf_{(t,x) \to (1,x_0)} u_2(t,x) \geq 0$. 
    \item[(ii)] For any $(t_0,x_0) \in [0,1) \times \pa \Delta$, $\liminf_{(t,x) \to (t_0,x_0)} u_2(t,x)=+\infty$
    \item[(iii)] For any convex test function $\phi \in C^2([0,1) \times \Delta; \R)$, and any local minimum $(t_0,x_0)$ of $u_2 - \phi$, we have 
    \begin{align*}
        -\pa_t \phi(t_0,x_0) \geq H(t_0, u_2(t_0,x_0), \det \nabla^2 \phi(t_0,x_0)).
    \end{align*}
\end{enumerate}
A continuous function $u: [0,1) \times \Delta \to \R$ is said to be a viscosity solution to \eqref{eq:gHJB} if it is a sub-solution and super-solution. 
\end{definition}

\begin{remark}
Viscosity solutions are widely used to tackle nonlinear degenerate elliptic /  parabolic PDEs; see e.g.\ \cite{Userguide} and the references therein. For control problems, it is standard that the associated value functions are the unique solutions to corresponding HJB equations \cite{Ph09,To13}. The challenge here is the infinite boundary condition in \eqref{eq:HJB}. 

The notion of viscosity solution for Monge-Amp\`{e}re equations was introduced by Ishii and Lions in \cite[Section V.3]{IsLi90}. As the function $\Sigma \mapsto \det \Sigma$ is monotone for positive semi-definite matrices, the test functions are restricted to the class of convex functions. 
\end{remark}

Now we verify the viscosity property of the value function. The proof is standard and relies on the dynamic programming principle of the value function. We record it here for the convenience of the reader.  The expert reader can skip it. 

\begin{proposition}\label{prop:viscosity}
 The convex value function $v:[0,1) \times \Delta \to \R$ satisfies the dynamic programming principle, i.e., for any $(t,x) \in [0,1) \times \Delta $
    \begin{align*}
v(t,x)=&\textstyle  \inf_{\mathbb Q \in \mathcal{M}_{t,x}^{win}} \inf_{\tau \in \mathcal{T}_t^1} \E_{\mathbb Q}\left[ \int_{t}^{\tau} - \log \det \Sigma_u \, du + v(\tau,X_{\tau}) \right] \\
=& \textstyle \inf_{\mathbb Q \in \mathcal{M}_{t,x}^{win}} \sup_{\tau \in \mathcal{T}_t^1} \E_{\mathbb Q}\left[ \int_{t}^{\tau} - \log \det \Sigma_u \, du + v(\tau,X_{\tau}) \right],
    \end{align*}
    where $(X_t)_{t \in [0,1]}$ is the canonical process defined on the path space, and $\mathcal{T}_t^1$ denotes the set of stopping times with respect to the canonical filtration taking values in $[t,1]$. Moreover, the value function $v(t,x)$ is a viscosity solution to \eqref{eq:HJB}.
\end{proposition}
\begin{proof}
According to the comment at the beginning of Section~\ref{sec3}, it is equivalent to consider the strong formulation of this control problem. In this case, the proof of the dynamic programming principle can be found in, for example, \cite[Subsection 3.3]{Ph09}. According to Lemma~\ref{lem:boundary-terminal}, $v(t,x)$ satisfies the boundary and terminal conditions of \eqref{eq:HJB}. It remains to verify its viscosity property. 

Suppose $\phi:[0,1) \times \Delta \to \R$ is a convex test function such that $v-\phi$ obtains a local maximum at $(t_0,x_0)$ with $v(t_0,x_0)=\phi(t_0,x_0)$, and we aim to show that 
\begin{align*}
   \textstyle  -\pa_t \phi(t_0,x_0) \leq \log \det \left(\frac{1}{2} \nabla^2 \phi(t_0,x_0) \right)+d. 
\end{align*}
Take a small $r>0$ such that  $B_r(t_0) \times B_r(x_0) \subset [0,1) \times \Delta$ and $v \leq \phi$ over $B_r(t_0) \times B_r(x_0)$, where $B_r(t_0)$ and $B_r(x_0)$ denote open Euclidean balls with radius $r$ and centers $t_0$ and $x_0$ respectively. Define $\tau:= \inf\{s \geq t_0: \, X_s \not \in B_r(x_0), \, X_{t_0}=x_0 \}$. For $n \in \mathbb N$, thanks to the dynamic programming principle 
\begin{align*}
 \phi(t_0,x_0)=v(t_0,x_0) \leq &\textstyle  \inf_{\mathbb Q \in \mathcal{M}_{t,x}^{win}} \E_{\mathbb Q} \left[\int_{t_0}^{(t_0+1/n)\wedge \tau} -\log \det \Sigma_u \, du + v((t_0+1/n)\wedge \tau,X_{(t_0+1/n)\wedge \tau}) \right] \\ 
 \leq &\textstyle  \inf_{\mathbb Q \in \mathcal{M}_{t,x}^{win}} \E_{\mathbb Q} \left[\int_{t_0}^{(t_0+1/n)\wedge \tau} -\log \det \Sigma_u \, du + \phi((t_0+1/n)\wedge \tau,X_{(t_0+1/n)\wedge \tau}) \right].
\end{align*}
For any $\Sigma^0 \in \mathbb S^d$, we construct $\mathbb Q_{\Sigma_0}$ to be the distribution of $dX_t=\Sigma^0 \, dB_t$, $X_{t_0}=x_0$ for $t \in [t_0,\tau]$, and a win-martingale from position $(\tau,X_{\tau})$ after $\tau$. Therefore, we have 
\begin{align*}\textstyle 
    \phi(t_0,x_0) \leq \E_{\mathbb Q_{\Sigma_0}} \left[\int_{t_0}^{(t_0+1/n)\wedge \tau} -\log \det \Sigma^0 \, du + \phi((t_0+1/n)\wedge \tau,X_{(t_0+1/n)\wedge \tau}) \right].
\end{align*}
An application of It\^{o}'s formula yields that, for any $n \in \mathbb N$
\begin{align*}\textstyle 
  0 \leq n\E^{\mathbb Q_{\Sigma_0}} \left[\int_{t_0}^{(t_0+1/n)\wedge \tau}- \log\det \Sigma^0 + \frac{1}{2}\Tr(\Sigma^0 \nabla^2 \phi(u,X_u))+\pa_t \phi(u,X_u) \, du \right].
\end{align*}
It can be easily checked that $n ( (t_0+1/n)\wedge \tau-t_0) \to 1$ $a.s.$, and $\pa_t \phi$, $\nabla^2\phi$ are uniformly bounded over $[t_0,t_0+1/n] \times B_r(x_0)$. Therefore we get that, for any $\Sigma^0 \in \mathbb S^d$
\begin{align*}
    0 \leq &\textstyle  \lim\limits_{n \to \infty} n\E_{\mathbb Q_{\Sigma_0}} \left[\int_{t_0}^{(t_0+1/n)\wedge \tau}- \log\det \Sigma^0 + \frac{1}{2}\Tr(\Sigma^0 \nabla^2 \phi(u,X_u))+\pa_t \phi(u,X_u) \, du \right] \\
    =&\textstyle -\log\det \Sigma^0+\frac{1}{2}\Tr(\Sigma^0 \nabla^2\phi(t_0,x_0)+ \pa_t \phi(t_0,x_0), 
\end{align*}
which implies 
\begin{align*}
-\pa_t \phi(t_0,x_0) \leq& \textstyle  \inf_{\Sigma^0 \in \mathbb S^d} \left(-\log\det\Sigma^0+\frac{1}{2}\Tr(\Sigma^0 \nabla^2\phi(t_0,x_0))  \right) \\
= &\textstyle  \log\det\left(\frac{1}{2} \nabla^2\phi(t_0,x_0)\right)+d, 
\end{align*}
since $\pa_t \phi(t_0,x_0)$ is finite and hence $\nabla^2\phi(t_0,x_0)$ must be necessarily positive definite.

The proof for the super-solution part is almost the same, and thus we skip it here. 
\end{proof}

At the end of this section, we will prove the uniqueness. Suppose $u_1$, $u_2$ are convex sub- and super-solutions resp. Then uniqueness follows from the comparison principle, stating that $u_1 \leq u_2$, as can be shown by comparing the derivatives of $u_1$ and $u_2$ at the maximum of $u_1-u_2$; see e.g.\ \cite{usersguide}. However, due to the infinite boundary condition, the maximum of $u_1-u_2$ may not be obtained in the interior of the domain. To address this, we use an observation from \cite{GuJi04}: For $\lambda>1$, scaled sub-solutions over space variables, $u_1^{\lambda}(t,x):=u_1(t,x/\lambda)$, are still sub-solutions up to a small error, and are now finite over $[0,1] \times \bar \Delta$. By standard doubling variables, we prove that  $u_1^{\lambda} \leq u_2$, and conclude  by $\lim_{\lambda \to 1} u_1^{\lambda}=u_1$.

\begin{proposition}\label{prop:uniqueness}
For \eqref{eq:HJB}, any viscosity sub-solution $u_1$ is smaller than any viscosity super-solution $u_2$. Therefore $v$ defined via \eqref{eq:optimal} is the unique viscosity solution to \eqref{eq:HJB}.
\end{proposition}
\begin{proof}
    \noindent \emph{Step 1:} Without loss of generality, we assume that $0 \in \Delta$. Indeed, for any $x_0 \in \Delta$, define $\Delta^{x_0}=\{x-x_0: \, x \in \Delta \}$ and $u^{x_0}(t,x)=u(t,x+x_0)$ for any $(t,x) \in [0,1) \times \Delta^{x_0}$. Then it can be easily verified that $u$ is a viscosity sub(super)-solution to \eqref{eq:HJB} iff $u^{x_0}$ is a viscosity sub(super)-solution to 
    \begin{equation*}
\begin{cases}
-\pa_t v(t,x)= \log\det\left(\frac{1}{2} \nabla^2 v(t,x) \right)+d , \hspace{6pt} \quad (t,x) \in [0,1)\times \Delta^{x_0} \\
\quad  \, \ v(t,x)=+\infty, \quad \hspace{82pt}  (t,x) \in [0,1) \times \partial \Delta^{x_0}, \\
\quad  \  v(1,x)=0, \quad \quad \hspace{81pt} (1,x) \in \{1\} \times \Delta^{x_0}.
\end{cases}
\end{equation*}

\medskip 
\noindent \emph{Step 2a:} For any $\lambda>1$, define $\Delta_{\lambda}:= \{\lambda x: \, x \in \Delta \}$, and \[ u_1^{\lambda}(t,x):= u_1(t,x/\lambda)+2d(t-1) \log \lambda, \quad (t,x) \in [0,1)\times \Delta_{\lambda}.\] As $0 \in \Delta$, $\Delta$ is a proper subset of $\Delta_{\lambda}$, and hence $u_1^{\lambda}$ is well-defined over $[0,1] \times \overline{\Delta}$. We claim that it is a viscosity sub-solution to \eqref{eq:HJB}. 

It is straightforward that for any $x_0 \in \Delta$, $\limsup_{(t,x) \to (1,x_0)} u^{\lambda}_1(t,x) \leq 0$. Take any convex  $\phi \in C^2([0,1) \times \Delta; \R)$ such that $u_1^{\lambda}-\phi$ obtains a local maximum at $(t_0,x_0) \in [0,1) \times \Delta$. With 
\begin{align*}
    \phi^{\lambda}(t,x)=\phi(t,\lambda x)-2d(t-1) \log(\lambda), 
\end{align*}
$u_1-\phi^{\lambda}$ obtains a local maximum at $(t_0,x_0/\lambda)$. Since $u_1$ is a sub-solution, 
    \begin{align*}\textstyle 
        -\pa_t \phi^{\lambda}(t_0,x_0/\lambda) \leq  \log \det\left(\frac{1}{2} \nabla^2\phi^{\lambda}(t_0,x_0/\lambda) \right)+ d.
    \end{align*}
As $|\pa_t \phi^{\lambda}(t_0,x_0/\lambda)|$ is finite, it implies that $\nabla^2\phi^{\lambda}(t_0,x_0/\lambda)=\lambda^2 \nabla^2 \phi(t_0,x_0)>0$. Therefore 
\begin{align*}
    -\pa_t \phi(t_0,x_0)+2d\log \lambda= &\textstyle -\pa_t \phi^{\lambda}(t_0,x_0/\lambda) 
    \leq  \log \det\left(\frac{1}{2} \nabla^2\phi^{\lambda}(t_0,x_0/\lambda) \right)+ d \\
    =& \textstyle  \log \det\left(\frac{1}{2} \nabla^2\phi(t_0,x_0) \right)+ 2d \log \lambda + d,
\end{align*}
which justifies the sub-solution property of $u_1^{\lambda}$.

\medskip 
\noindent \emph{Step 2b:} Let us define 
\begin{align*}
    \hat u_1^{\lambda}(t,x):=e^t u_1^{\lambda}(t,x), \quad \hat u_2(t,x):=e^t u_2(t,x), \quad (t,x) \in [0,1) \times \Delta. 
\end{align*}
Then by the same argument as in \emph{Step 2a:}, it can be easily verified that $\hat u_1^{\lambda}$ and $\hat u_2$ are respectively a sub-solution and super-solution to 
\begin{equation}
\begin{cases}
    -\pa_t u =-u +e^t \log \det \nabla^2 u+ e^t d(1-t), \quad \hspace{30pt} (t,x) \in [0,1) \times \Delta,\\
   \quad  u=+\infty, \quad \hspace{142pt} (t,x) \in [0,1) \times \pa \Delta  ,\\
   \quad   u=0, \quad \quad \hspace{142pt}  (1,x) \in  \{1\} \times \Delta.
    \end{cases}
\end{equation}

\medskip
\noindent \emph{Step 3a:}
Let us show that $\hat u_1^{\lambda} \leq \hat u_2$ over $[0,1) \times \Delta$. If not, we consider the maximization
\begin{align*}
    \argmax_{(t,x) \in [0,1) \times \Delta} \hat u_1^{\lambda}(t,x)-\hat u_2(t,x). 
\end{align*}
Note that $\hat u_1^{\lambda}$ is well-defined over $[0,1) \times \Delta_{\lambda}$, and hence according to the definition of viscosity sub-solution is continuous and bounded over $[0,1] \times \overline{\Delta}$. Therefore for any $(t_0,x_0) \in [0,1) \times \pa \Delta$, 
\[ \lim\limits_{(t,x) \to (t_0,x_0)} \hat u_1^{\lambda}(t,x)- \hat u_2(t,x)=-\infty, \]
and for any $x_0 \in \overline{\Delta}$, 
\[ \limsup\limits_{(t,x) \to (1,x_0)} \hat u_1^{\lambda}(t,x)- \hat u_2(t,x) \leq 0.\]
As $\sup_{(t,x) \in [0,1) \times \Delta} \hat u_1^{\lambda}(t,x)-\hat u_2(t,x)=:\delta >0 $, the supremum must be obtained in the interior
\begin{align*}\textstyle 
    [0,1) \times \Delta \ni (t_0,x_0)=  \argmax_{(t,x) \in [0,1) \times \Delta} u_1^{\lambda}(t,x)-u_2(t,x). 
\end{align*}

\medskip
\noindent \emph{Step 3b:} W.l.o.g., we assume that $\hat u_2=+\infty$ outside a closed small ball of $(t_0,x_0)$ as such modification still preserves its super-solution property. For $\alpha>0$, let us consider 
\begin{align}
    \Phi_{\alpha}(t,x,y):= \hat u_1^{\lambda}(t,x) -\hat u_2(t,y) -\frac{\alpha}{2}|x-y|^2. 
\end{align}
We maximize $\Phi_{\alpha}$ over $[0,1)\times \Delta \times \Delta$, where the maximum $(t_{\alpha},x_{\alpha},y_{\alpha})$ must be obtained in the interior thanks to the modification of $\hat u_2$, i.e., 
$$\textstyle M_\alpha:=\sup_{(t,x,y) \in [0,1) \times \Delta \times \Delta} \Phi_{\alpha}(t,x, y)=\Phi_{\alpha}(t_{\alpha},x_{\alpha},y_{\alpha})$$
It can be easily checked that $\lim_{\alpha \to \infty} M_{\alpha}=\delta$, $\lim_{\alpha \to \infty}\alpha|x_{\alpha}-y_{\alpha}|^2 \to 0$, and an accumulating point of $(t_{\alpha},x_{\alpha},y_{\alpha})$ as $\alpha \to \infty$ is a maximizer of $(t,x) \mapsto \hat u_1^{\lambda}-\hat u_2$; see e.g. \cite[Lemma 3.1]{Userguide}. Without loss of generality, we assume that 
\[ (t_\alpha,x_\alpha,y_\alpha) \to (t_0, x_0 ,x_0) \ \  \text{as } \alpha \to \infty. \]

Applying the parabolic version of Ishii's Lemma \cite[Theorem 8.3]{Userguide} to $\Phi_{\alpha}$, due to the continuity of $\R \times \mathbb S^d \ni (b,\Sigma) \mapsto (b,\log \det \Sigma)$, there exist $b_{\alpha} \in \R$, $\Sigma_{\alpha}, \Sigma_{\alpha}' \in \mathbb S^d$ such that  $\Sigma_{\alpha} \leq \Sigma_{\alpha}'$ and 
\begin{align*}
   -b_{\alpha} \leq& - \hat u_1^{\lambda}(t_\alpha,x_\alpha)+e^{t_\alpha} \log \det \Sigma_{\alpha}+e^{t_\alpha}(1-t_\alpha), \\
   -b_{\alpha} \geq& -\hat u_2(t_\alpha,y_\alpha)+e^{t_\alpha} \log \det \Sigma_{\alpha}'+e^{t_\alpha}(1-t_\alpha). 
\end{align*}
The above implies the inequality 
\begin{align*}
    \hat u_1^{\lambda}(t_\alpha,x_\alpha) &\leq e^{t_\alpha} \log \det \Sigma_{\alpha}+b_{\alpha}+e^{t_{\alpha}}(1-t_{\alpha})  \\
    &\leq e^{t_\alpha} \log \det \Sigma_{\alpha}'+b_{\alpha}+e^{t_{\alpha}}(1-t_{\alpha})  \leq \hat u_2(t_\alpha,y_\alpha) 
\end{align*}
for every large $\alpha>0$, which contradicts with the fact that 
\begin{align*}
    \lim\limits_{\alpha \to \infty}   \hat u_1^{\lambda}(t_\alpha,x_\alpha)-\hat u_2(t_\alpha,y_\alpha) = \delta >0.
\end{align*}

\medskip 
\noindent \emph{Step 4:} We have shown that $\hat u_1^{\lambda} \leq \hat u_2$ over $[0,1) \times \Delta$ for any $\lambda>1$. Taking $\lambda \to 1$ and using the continuity of $\lambda \to \hat u^{\lambda}_1(t,x)$, we conclude the result. 
\end{proof}

\section{ Existence and asymptotic behavior of  the Monge-Amp\`{e}re Equation \eqref{eq:Monge}  }\label{sec5}

In Lemma~\ref{lem:boundary-terminal} we obtained the decomposition of the value function $v(t,x)=(1-t)v(0,x)+d(1-t) \log(1-t)$,  so  \eqref{eq:HJB} yields that 
\begin{equation}\label{eq:fg}\textstyle 
v(0,x)+d \log(1-t)+d=  \log \det\left(\frac{1}{2}\nabla^2 v(0,x) \right)+d \log (1-t)+d.
\end{equation}
 Therefore, the function $\Delta \ni x \mapsto g(x):=v(0,x)$ satisfies \eqref{eq:Monge}.

In this section, we prove the first main result of the paper: Theorem~\ref{thm:quantitive}. {To make the paper self-contained, first we give an alternative proof for the existence of a smooth solution to \eqref{eq:Monge}, which was already established in \cite{GuJi04}.} Then we prove the novel parts of Theorem \ref{thm:quantitive}, concerning the asymptotic behavior of $\nabla^2 g$ near the boundary $\partial \Delta$.

The key ingredient of the arguments is the following barrier function of \eqref{eq:Monge}:
\begin{align}\label{eq:w}\textstyle 
w(x):=-2 \sum_{i=1}^d \log(x_i)-2 \log\left(1-\sum_{i=1}^d x_i \right), \, \quad x \in \Delta.
\end{align}

\begin{lemma}\label{lem:maincomponent}
   The functions $w-\log(d+1)$ and $w$ are smooth sub- and super-solution to \eqref{eq:Monge} respectively. Furthermore, we have  that $\inf_{x\in\Delta} \{w(x)-\log(d+1)\}> 0$.
  \end{lemma}
\begin{proof}
For simplicity, we denote $w_i(x)=\pa_{x_i} w(x)$ and $w_{ij}(x)=\pa^2_{x_ix_j}w(x)$ for any $1\leq i,j \leq d$. By direct computation, 
\begin{align*}
    w_{ij}(x)&\textstyle = \frac{2\delta_{ij}}{x_i^2}+\frac{2}{(1-\sum_{k=1}^d x_k)^2}, \\
    \nabla^2 w(x)&\textstyle =Diag\left(2/x_i^2\right)_{i=1}^d+ \frac{2}{(1-\sum_{k=1}^d x_k)^2} \mathbf{1} \mathbf{1}^\top,
\end{align*}
where $Diag\left(2/x_i^2\right)_{i=1}^d$ denotes the diagonal matrix with the $(i,i)$ entry being $\frac{1}{x_i^2}$ for $i=1,\dotso, d$. Therefore thanks to Sylvester's determinant identity, we get that 
\begin{align*}
   \textstyle  \det\left(\frac{1}{2}\nabla^2 w(x) \right)= \prod_{i=1}^d\frac{1}{x_i^2} \left(1+\frac{\sum_{i=1}^d x_i^2}{(1-\sum_{i=1}^d x_i)^2} \right).
\end{align*}
It is straightforward that $e^{w(x)}=\frac{1}{(1-\sum_{i=1}^d x_i)^2}\prod_{i=1}^d \frac{1}{x_i^2}$, 
and therefore
\begin{align*}
 \textstyle    \frac{1}{d+1} \leq \frac{\det\left(\frac{1}{2}\nabla^2 w(x) \right)}{e^{w(x)}}= \sum_{i=1}^d x_i^2+ \left(1-\sum_{i=1}^d x_i\right)^2 \leq 1.
\end{align*}
Since $\lim_{x \to \partial\Delta} w(x)=+\infty$, the inequality above concludes our first claim. Finally, as $\R_+ \ni z \mapsto -\log(z)$ is a convex function, Jensen's inequality shows that 
\begin{align*}
    w(x) \geq- 2(d+1) \log( 1/(d+1))=2(d+1)\log(d+1) > \log(d+1),
\end{align*}
and thus $\inf_{x\in\Delta} \{w(x)-\log(d+1)\}> 0$.
\end{proof}

\subsection{Existence} 
In this subsection, we establish the  existence of \eqref{eq:Monge} in Proposition~\ref{prop:existence}, and show that $\nabla g: \Delta \to \mathbb R^d$ is a diffeomorphism in Lemma~\ref{lem:diffeomorphism}.

\begin{proposition}\label{prop:existence}
    There exists a smooth solution $g$ to \eqref{eq:Monge}. 
\end{proposition}

The proof is divided into several steps. For any $n \in \mathbb N$, we denote $\Omega_n=\{x \in \Delta : \, w(x) < n \}$. It is clear that $\cup_{n \in \mathbb N} \Omega_n= \Delta$. Due to the explicit formula of $w$ \eqref{eq:w}, $\Omega_n$ is strictly convex with smooth boundary. Therefore, according to \cite{CNS1}, for each $n$ there exists a smooth solution $g^n$ to the elliptic equation 
\begin{align}\label{eq:restriction}
    \begin{cases}
        g^n(x)= \log \det \left( \frac{1}{2} g^n(x)\right),  \hspace{1.4cm} x \in \Omega_n, \\
        g^n(x)= w(x), \hspace{3cm} x \in \partial \Omega_n. 
    \end{cases}
\end{align}
As in Lemma~\ref{lem:maincomponent}, it can be verified that $w|_{\Omega_n}$ is a supersolution to \eqref{eq:restriction}
while $w|_{\Omega_n}-\log(d+1)$ is a subsolution to \eqref{eq:restriction}. Therefore, the comparison principle for elliptic equations yields that $w - \log(d+1) \leq g^n \leq w \leq n$ over $\Omega_n$.

 Using the barrier function constructed in  Lemma~\ref{lem:maincomponent}, we derive local $C^1, C^2$ estimates in Lemma~\ref{lem:e1} and Lemma~\ref{lem:e2} respectively, and therefore get the compactness of $\{g^n\}_{n \in \mathbb N}$. We prove that a limit of $\{g^n\}_{n \in \mathbb N}$ is a solution to \eqref{eq:Monge}, and show the uniqueness in Corollary~\ref{cor:val_function_fg}.

\begin{lemma}\label{lem:e1}
    For any $N \in \mathbb N$, we have that
\begin{align}\label{eq:C2estimate}
\lVert g^n \rVert_{C^1(\bar \Omega_N)} \leq K(N) \quad \text{for all $n \geq N+10$},
\end{align}
where $K(N)$ is a positive constant depending only on $N$.
\end{lemma}
\begin{proof}
    From the construction of $g^n$ and the comparison principle,  $g^n(x) \leq w(x) \leq N+10$ over $\bar \Omega_N$. Let us estimate the first order derivative.

    Fix any $\xi \in \mathcal{S}^{d-1}$ and take $C=N+10$. For any $n \geq N+10$, we consider the problem 
    \begin{align*}
    \textstyle     \max_{x \in \bar\Omega_C}(C-w(x)) \, g^n_{\xi}(x) \, e^{g^n(x)},
    \end{align*}
    where a maximizer, denoted by $x_0$, must be obtained in the interior $\Omega_C$ since $C-w=0$ on the boundary $\partial \Omega_C$. It is clear that $g^n_{\xi}(x_0) >0$. Let us take 
    \begin{align*}
        \rho(x)= \log (C-w(x))+\log(g^n_{\xi}(x)) +g^n(x),
    \end{align*}
    which also obtains the maximum at $x_0$. The first order condition in direction $\xi$ yields that 
    \begin{align*}
        \textstyle 0=\frac{-w_{\xi}(x_0)}{C-w(x_0)}+\frac{g^n_{\xi\xi}(x_0)}{g^n(\xi)(x_0)}+g^n_{\xi}(x_0),
    \end{align*}
    and hence at $x_0$, $g^n_{\xi}(x_0) \leq \frac{w_{\xi}(x_0)}{C-w(x_0)} $. Therefore for any $x \in \Omega_C$, we have 
    \begin{align*}
        (C-w(x)) \, g^n_{\xi}(x) \, e^{g^n(x)} \leq (C-w(x_0) \, g^n_{\xi}(x_0) \, e^{g^n(x_0)} \leq w_{\xi}(x_0) \, e^{g^n(x_0)}.
    \end{align*}
    Restricting $x \in \Omega_N$ on the left hand side, we have that $(C-w(x)) \geq 10$, and $g^n(x) \geq w(x)-\log(d+1) \geq \inf_{y \in \Delta} w(y)- \log(d+1)>0 $ from Lemma~\ref{lem:maincomponent}. Together with $g^n(y) \leq N+10$ for $y \in \Omega_C$, 
    \begin{align*}
     \textstyle  \sup_{x \in \bar\Omega_N}  g_{\xi}^n(x) \leq K(N) \sup_{y \in \Omega_C}  w_{\xi}(y). 
    \end{align*}
    Note that the right hand side is independent of $n$.

    To conclude, $w(y)\leq C $ implies $y_i \geq e^{-C/2}, (1-\sum_{j=1}^d y_j) \geq e^{-C/2}$, $i=1,\dotso, d$, and hence 
    \begin{align*}
      \textstyle   \left| \partial_{i} w(y)=\frac{2}{1-\sum_{j=1}^d y_j}-\frac{2}{y_i} \right| \leq 2e^{C/2}=2 e^{(N+10)/2} \quad \text{for all $y \in \Omega_C$}. 
    \end{align*}
  
\end{proof}

\begin{lemma}\label{lem:e2}
     For any $N \in \mathbb N$, we have that
\begin{align}\label{eq:C2estimate}
\lVert g^n \rVert_{C^2(\bar \Omega_N)} \leq K(N) \quad \text{for all $n \geq N+10$},
\end{align}
where $K(N)$ is a positive constant depending only on $N$.
\end{lemma}
\begin{proof}
In Lemma~\ref{lem:e1}, we already have  $\lVert g^n \rVert_{C^1(\bar \Omega_N)} \leq K(N)$. We prove that for any $\xi \in \mathcal{S}^{d-1}$, $$\textstyle \sup_{x \in \bar \Omega_N}|g^n_{\xi \xi}(x)| \leq K(N).$$ 

In the rest of proof, let us omit the superscript $n$ from $g^n$. Take $C=N+10$, we consider the maximization problem 
\begin{align*}
   \textstyle  \sup_{x \in \bar \Omega_C, \, \xi \in \mathcal{S}^{d-1}} (C-g(x))^{2d}\, g_{\xi \xi}(x) \, e^{|x|^2/4}.
\end{align*}
Without loss of generality, we assume the maximizer on the left hand is obtained at $x_0 \in \Omega_C$, and $\xi=e_1$. Let us denote 
\begin{align*}
    \rho(x)=2d \log(C-g(x))+ \log(g_{11}(x))+ \frac{|x|^2}{4}.
\end{align*}
Since $\mathbb R_+ \ni r \mapsto \log r$ is an increasing function, $x_0$  is a also maximizer of $\max_{x \in  \bar\Omega_C} \rho(x)$.  Choosing another orthonormal coordinate system $(x_1,\dotso,x_d)$ of $\mathbb R^d$ with $e_1$ being the direction of $x_1$, without loss of generality, we can assume that $\nabla^2 g(x_0)$ is diagonal, and denote $\lambda_i= g_{ii}(x_0)$ for each $i=1,\dotso, d$. Let us also denote by $(g^{ij})_{i,j=1}^d$ the inverse matrix of $\nabla^2 g(x_0)$. From now on, all the derivatives are taken at $x_0$, and we adopt the Einstein summation notation.

 Taking derivatives on both sides of the Monge-Amp\`{e}re $g= \log \det \left(\frac{1}{2} \nabla^2 g\right)$, we get that 
\begin{align}\label{eq:2orderMA'}
   g_i&= g^{lk}g_{lki}, \quad  \quad\quad\quad\quad\quad\quad \ \ \ i=1,\dotso, d,  \notag \\
   g_{ij}&=g^{lk}g_{lkij}-g^{lm}g^{kn}g_{mnj}g_{lki}, \quad i,j=1,\dotso, d.
\end{align}
As $\rho(x_0)=\sup_{x \in \bar \Omega_C} \rho(x)$, the first order condition implies that 
\begin{align}\label{eq:1order'}
   \textstyle  0=\rho_i=-\frac{2d g_i}{C-g}+\frac{g_{11i}}{g_{11}}+\frac{x_i}{2},  \quad  \quad i=1,\dotso,d.
\end{align}
Taking derivative again, 
\begin{align*}
   \textstyle  \rho_{ij}=-\frac{2d g_{ij}}{C-g}-\frac{2d g_ig_j}{(C-g)^2}+\frac{g_{11ij}}{\lambda_1}-\frac{g_{11i}g_{11j}}{\lambda_1^2} +\frac{1}{2} \delta_{ij}, \quad \quad i,j=1,\dotso, d,
\end{align*}
where $\delta_{ij}$ is the Kronecker symbol. The maximality condition implies that as a matrix $\nabla^2 \rho$ is negative semidefinite, and hence due to the convexity of $g$
\begin{align*}
    \textstyle 0 \geq g^{ij} \rho_{ij}=- \frac{2d^2}{C-g}-\sum_{i=1}^d \frac{2d g_i^2 }{\lambda_i (C-g)^2}+\frac{g^{ii }g_{11ii}}{\lambda_1}-\sum_{i=1}^d \frac{g_{11i}^2}{\lambda_1^2\lambda_i}+\sum_{i=1}^d \frac{1}{2\lambda_i }.
\end{align*}

Setting  $i=j=1$ in \eqref{eq:2orderMA'}, we get that $g^{ii}g_{11ii}=g_{11}+g^{ll}g^{kk}g_{lk1}^2$, and hence 
\begin{align*}
    0 \geq &\textstyle  - \frac{2d^2}{C-g}-\sum_{i=1}^d \frac{2d g_i^2 }{\lambda_i (C-g)^2}+1+\sum_{l,k=1}^d \frac{g_{lk1}^2}{\lambda_1\lambda_l \lambda_k}-\sum_{i=1}^d \frac{g_{11i}^2}{\lambda_1^2\lambda_i}+\sum_{i=1}^d \frac{1}{2\lambda_i } \\
    \geq &\textstyle  - \frac{2d^2}{C-g}-\sum_{i=1}^d \frac{2d g_i^2 }{\lambda_i (C-g)^2}+1+\sum_{k=2}^d \frac{g_{11k}^2}{\lambda_1^2 \lambda_k}+\sum_{i=1}^d \frac{1}{2\lambda_i }.
\end{align*}
According to \eqref{eq:1order'}, it can be easily seen that for $k=2,\dotso, d$, 
\begin{align*}
  \textstyle   -\frac{2d g_k^2 }{\lambda_k (C-g)^2}=-\frac{1}{2d \lambda_k}\left(\frac{g_{11k}}{\lambda_1}+ \frac{x_k}{2} \right)^2 \geq -\frac{1}{d \lambda_k}\left(\frac{g_{11k}^2}{\lambda_1^2}+\frac{  x_k^2}{4} \right).
\end{align*}
Therefore, by direct computation we get that 
\begin{align}\label{eq:step1'}
   \textstyle  \frac{2d^2}{C-g} +\frac{2d g_1^2 }{\lambda_1 (C-g)^2} \geq 1+\sum_{k=2}^d \left( \frac{g_{11k}^2}{\lambda_1^2 \lambda_k}-\frac{1}{d}\frac{g_{11k}^2}{\lambda_1^2 \lambda_k}\right) +\sum_{i=1}^d \frac{1}{4\lambda_i} \geq \sum_{i=1}^d \frac{1}{4\lambda_i}. 
\end{align}

Let us prove a lower bound for $\lambda_1$ using the Monge-Amp\`{e}re equation. Note that $$\lambda_1^d \geq \lambda_1 \dotso \lambda_d=\det \nabla^2 g(x_0)=2^d e^{g(x_0)}\geq 2^d e^{ w (x_0)- \log(d+1)} \geq 2^d,$$
according to Lemma~\ref{lem:maincomponent}. Hence we conclude that $\lambda_1 \geq K$ for an absolute constant $K$ which is allowed to change from line to line, and together with \eqref{eq:step1'},
\begin{align*}\textstyle 
\frac{1}{C-g} +\frac{ g_1^2 }{(C-g)^2} \geq K \sum_{i=1}^d \frac{1}{\lambda_i}.
\end{align*}

Now, we would like to derive an upper bound of $\lambda_1$,
\begin{align*}\textstyle 
    \sum_{i=1}^d \frac{1}{\lambda_i} \geq \sum_{i=2}^d \frac{1}{\lambda_i} \geq (d-1) \sqrt[d-1]{\frac{1}{\lambda_2 \dotso \lambda_d }}=(d-1) \sqrt[d-1]{\frac{\lambda_1}{\det \nabla^2 g }} =(d-1) \sqrt[d-1]{\frac{\lambda_1}{2^d e^{g} }}.
\end{align*}
Therefore it can be easily seen that 
\begin{align*}\textstyle 
   \lambda_1 \leq e^g K \left(\frac{1}{C-g} +\frac{ g_1^2 }{(C-g)^2} \right)^{d-1},
\end{align*}
and hence together with Lemma~\ref{lem:e1}
\begin{align*}
   \left(C-g(x_0) \right)^{2d} \, \lambda_1 \, e^{\frac{|x_0|^2}{4}} \leq K(C) \left(1+ g_1^{2d-2}\right) \leq K(N).
\end{align*}
As a result,  $\max_{x \in \bar \Omega_N, \, \xi \in \mathcal{S}^{d-1}} (C-g(x))^{2d} \,  g_{\xi\xi}(x) \, e^{\frac{|x|^2}{4}} \leq  K(N)$. Since $C=N+10$, $g(x) \leq N$ for any $x \in \Omega_N$, and $e^{\frac{|x|^2}{4}} \geq 1$, we conclude that $g_{\xi\xi}(x) \leq K(N)$ for every $x \in \bar\Omega_N$. 

\end{proof}

\begin{proof}[Proof of Proposition~\ref{prop:existence}]
Fixing $\alpha \in (0,1)$, according to the Evans–Krylov theorem, see  \cite{CaCa95} and \cite[Theorem A.42]{Fi17}, Lemma~\ref{eq:2orderMA'} yields that, for each $N \in \mathbb N$,
\begin{align*}
    \lVert g^n \rVert_{C^{2,\alpha}(\bar \Omega_N)} \leq K(N) \quad \text{for any $n \geq N+10$}.
\end{align*}
Since $C^{2,\alpha}(\bar \Omega_N)$ is compactly embedded in $C^{2}(\bar \Omega_N)$, there exists a subsequence of $(g^{n_k})_{k \in \mathbb N}$ and a $C^2$ function $g:\Delta \to \mathbb R$  such that $\lim\limits_{k \to \infty}\lVert g^{n_k} -g \lVert_{C^2(\bar \Omega_N)} \to 0$ for each $N \in \mathbb N$. Then it can be easily seen that $g$ is a solution to \eqref{eq:Monge}.
\end{proof}

Denoting $f(t)=d(1-t)\log(1-t)$,  Proposition~\ref{prop:uniqueness} and Proposition~\ref{prop:existence} yield that 

\begin{corollary}\label{cor:val_function_fg}
$v(t,x)= (1-t)g(x)+f(t)$, $\forall \, (t,x) \in [0,1) \times \Delta$. Moreover, $g$ is the unique solution to \eqref{eq:Monge}, and we have the inequality $w(x) -\log(d+1) \leq g(x) \leq w(x)$, $\forall \, x \in \Delta$.
\end{corollary}
\begin{proof}
 By Proposition \ref{prop:uniqueness}, we know that $v$ defined via \eqref{eq:optimal} is the unique viscosity solution to \eqref{eq:HJB}, while on the other hand $(1-t)g(x)+f(t)$ is clearly a smooth solution to the same equation. It also implies the uniqueness of $g$.  The inequality directly follows from $w-\log(d+1) \leq g^n \leq w$ for any $n$ and the fact that $g^n \to g$.
\end{proof}

We also have some soft properties of the solution $g$ to \eqref{eq:Monge}, for instance

\begin{lemma}\label{lem:diffeomorphism}
    $\nabla g: \Delta  \to \mathbb R^d $ is a diffeomorphism.
\end{lemma}

\begin{proof}
    As $g$ is smooth and $\nabla^2 g$ is invertible, $\nabla g: \Delta  \to \mathbb R^d $ is a local diffeomorphism. Hence $\nabla g(\Delta)$ is open. We now show that $\nabla g(\Delta)$ is closed, so  $\nabla g(\Delta)=\mathbb R^d$. It suffices to show that if $\nabla g(x_n)\to y$, then $\nabla g(\bar x)=y$ for some $\bar x\in \Delta$. By compactness we assume that $x_n\to x\in\overline \Delta $, so we only need to exclude the possibility that $x\in\partial \Delta$. By convexity we have
    \[g(z)\geq g(x_n)+ \nabla g(x_n)^{\top}(z-x_n), \quad  \forall \, z \in \Delta.\]
Since $\nabla g(x_n)^{\top}(z-x_n)$ is bounded, we see that if $x\in\partial \Delta$ then $g(z)=+\infty$ for all $z$, which is clearly a contradiction. To wit, $g(x_n)\to \infty$ as $n\to\infty$, as follows from Lemma~\ref{lem:maincomponent} and the boundary behaviour of the function $w$ therein.

Note that $g$ is strictly convex on $\Delta$ as $\det \nabla^2 g(x) >0$ for all $x \in \Delta$. Thus for any $x,y \in \Delta$
\begin{align*}
   g(x) > & g(y) + \nabla g(y)^{\top}(x-y), \\
   g(y) >&  g(x)+\nabla g(x)^{\top}(y-x), 
\end{align*}
    which implies that $ (\nabla g(x) -\nabla g(y))^{\top}(x-y)>0$ and hence $\nabla g(x) \not = \nabla g(y)$ if $x\neq y$. Thus $\nabla g$ is a global diffeomorphism from $\Delta$ to $\mathbb R^d $.
\end{proof}

\subsection{Boundary behavior}

In this subsection, we show the asymptotic behavior of $\nabla^2g$ near the boundary. It is inspired by the proof of Caffarelli and Li \cite{MR1953651}, and based on our observation of an invariant property of the equation \eqref{eq:Monge}: scaled smooth solutions remain solutions to essentially the same equation. Therefore we establish local $C^1,C^2$ estimates using scaled barrier functions to obtain the asymptotic behavior of $\nabla g$ and $\nabla^2 g$.

We make some simplification before proving Theorem~\ref{thm:quantitive}. Suppose $x^0=(x^0_1,\dotso,x^0_d)$ is a point in a $(d-k)$-dimensional face but not in a $(d-k-1)$-dimensional face with $1 \leq k \leq d$. W.l.o.g., we can assume the origin $\mathbf{0}$ is one vertex of this face. Otherwise, we do linear transformation and obtain the same equation. Indeed, if the origin is not one vertex of the face, we must have that $\sum_{i=1}^d x_i^0=1$. As $x^0$ is in the interior of a $(d-k)$-dimensional face, w.l.o.g., we assume that $\sum_{i=1}^d x^0_i=1$, $x^0_1=\dotso=x^0_{k-1}=0$, $x^0_{k},\dotso, x^0_d>0$. Let us define a linear transformation $L$ from $\Delta$ to $\Delta$ via 
\begin{align*}
y_i=&L_i(x)=x_i, \hspace{0.2cm} i=1,\dotso, k-1, k+1,\dotso, d, \\
y_k=&L_k(x)=1-x_1-x_2-\dotso-x_{d},  
\end{align*}
where $L_i(x)$ denotes the $i$-th component of $L(x)$, $i=1,\dotso, d$. It can be easily verified that $L$ is a one-to-one linear transformation from $\Delta$ to $\Delta$ such that $L=L^{-1}$ and $L^{-1}(x^0)=(0,\dotso,0,x^{0}_{k+1},\dotso, x^0_d)$ is in the interior of a $(d-k)$-dimensional with $\mathbf{0}$ as one of its vertices. Taking $\bar g(x):= g(L(x))$, since the determinant of the Jacobian of $L$ is $-1$, one can in fact show that $\bar g$ is a smooth solution to \eqref{eq:Monge} as well.

Therefore w.l.o.g., we assume  $x^0_1=\dotso=x^0_k=0$, $x^0_{k+1},\dotso,x^0_d > 0$, and $\sum_{j=k+1}^d x_j^0<1$. Theorem~\ref{thm:quantitive} \emph{(i)} will be a direct corollary of the following proposition. 

\begin{proposition}\label{prop:secondorder}
    Fix $x^0 \in \partial \Delta$ with $x^0_1=\dotso=x^0_k=0$,  $x^0_{k+1},\dotso,x^0_d > 0$, and $\sum_{j=k+1}^d x_j^0<1$. Taking $r= (1-\sum_{i=k+1}^d x^0_i)/{(2d)}$,  there exist two positive constants $c_1(x^0)$, $c_2(x^0)$ such that for any $x_1,\dotso, x_k \in (0,r)$,
    \begin{align*}
    c_1(x^0) Diag(1/x^2_1,\dotso,1/x^2_k, 1,\dotso, 1) 
  &\leq \nabla^2 g(x_1,\dotso, x_k, x^0_{k+1},\dotso,x^0_d) \\
  &\leq  c_2(x^0)Diag(1/x^2_1,\dotso,1/x^2_k, 1,\dotso ,1),
  \end{align*}
    where $Diag(1/x^2_1,\dotso,1/x^2_k, 1,\dotso, 1)$ denotes a diagonal matrix with the $(i,i)$-th entry being $1/x^2_i$ for $1\leq i \leq k$ and $1$ for $k+1 \leq i \leq d$. 
\end{proposition}
To prove this proposition, let us define some auxiliary functions. Denote $x^0_{k+1:d}:=(x^0_{k+1},\dotso, x^0_d)$. For $x_1,\dotso,x_k$ with $x_1,\dotso, x_k>0$ and $\sum_{i=1}^k x_i <1-\sum_{j=k+1}^d x^0_j$, similarly denote $x_{1:k}=(x_1,\dotso,x_k)$. We will analyze the Hessian of $g$ at $(x_{1:k},x_{k+1:d}^0)$ close to $x^0$. To this end, we define the linear transformation $T^{x_{1:k}}:\R^d \to \R^d$ 
\begin{align*}
    T^{x_{1:k}}: (y_1,\dotso, y_d) \mapsto (x_1(1+y_1), \dotso, x_k(1+y_k), y_{k+1}, \dotso, y_d), 
\end{align*}
in the domain $\Omega^{x_{1:k}}= (T^{x_{1:k}})^{-1}(\Delta)$, and the auxiliary function
\begin{align*}
    \eta^{x_{1:k}}: \Omega^{x_{1:k}} \to \mathbb R; \, \, y \mapsto g(T^{x_{1:k}}(y)).
\end{align*}
Then it can be easily checked that $\eta^{x_{1:k}}=\infty$ on $\pa \Omega^{x_{1:k}}$, and also
\begin{align}\label{eq:htransfer}
    \nabla^2_{ij} \eta^{x_{1:k}} (y)= 
    \begin{cases}
        x_i x_j \nabla^2_{ij} g(T^{x_{1:k}}(y)), \quad 1 \leq i,j \leq k, \\
        x_i \nabla^2_{ij} g(T^{x_{1:k}}(y)), \quad \ \ \  1 \leq i \leq k < j, \\
        \nabla^2_{ij}g(T^{x_{1:k}}(y)), \quad \quad \ \ \,  k < i,j.
    \end{cases}
\end{align}
Therefore we have the equality 
\begin{align*}
   \textstyle  \det \left(\frac{1}{2} \nabla^2 \eta^{x_{1:k}}(y) \right)&\textstyle =x_1^2\dotso x_k^2 \det \left(\frac{1}{2} \nabla^2 g(T^{x_{1:k}}(y)) \right)\\
    &\textstyle =x_1^2\dotso x_k^2 \, \exp({\eta^{x_{1:k}}(y)})= \exp\left({\eta^{x_{1:k}}(y)+2 \sum_{i=1}^k \log(x_i)}\right), 
\end{align*}
and the function 
\begin{equation}\label{eq:tileta}
\textstyle \tilde \eta^{x_{1:k}}(y):=\eta^{x_{1:k}}(y) + 2 \sum_{i=1}^k \log(x_i), \quad y \in \Omega^{x_{1:k}},
\end{equation}
satisfies the Monge-Amp\`{e}re 
\begin{align*} 
\begin{cases}
\tilde \eta^{x_{1:k}}=\log \det\left(\frac{1}{2}\nabla^2 \tilde \eta^{x_{1:k}}(x)\right), \quad y \in \Omega^{x_{1:k}}, \\
\tilde \eta^{x_{1:k}} =+\infty, \quad \quad \quad \ \ \quad \quad  \quad \quad y \in \partial \Omega^{x_{1:k}}. \notag
\end{cases}
\end{align*}
We will provide estimates of $\nabla^2 \tilde\eta^{x_{1:k}}(0,x^0_{k+1:d})$ uniformly for $x_{1:k}$. Then since $T^{x_{1:k}}(0,x^0_{k+1:d})= (x_{1:k}, x^0_{k+1:d})$, Proposition~\ref{prop:secondorder} will follow from the equality \eqref{eq:htransfer}.

Recalling the definition of $w$ in \eqref{eq:w}, we define 
\begin{align*}
   \textstyle  \tilde w^{x_{1:k}}(y):=w(T^{x_{1:k}}(y)) +  2 \sum_{i=1}^k \log(x_i), \, \quad y \in \Omega^{x_{1:k}}. 
\end{align*}
Thanks to Corollary~\ref{cor:val_function_fg}, $|\tilde w^{x_{1:k}}(y) - \tilde \eta^{x_{1:k}} (y)| \leq \log(d+1)$ for any $ y \in \Omega^{x_{1:k}}$. To restrict our argument on compact subsets of $y$, we define $h: \mathbb R^d \to \R; y \mapsto \sum_{i=1}^k 2y_i$ and consider, for any given $C>0$, the sublevel sets \[\tilde \Omega_C^{x_{1:k}}:=\left\{y \in \Omega^{x_{1:k}}: \, \tilde \eta^{x_{1:k}}(y)+h(y) < C\right\}.\]   The next lemma provides bounds for $y \in \tilde \Omega_C^{x_{1:k}}$.

Keep in mind that the definition of all the $T,\eta,\tilde\eta,\tilde w, \Omega, \tilde \Omega_C$ are dependent on $x_{1:k}$.  In Lemma~\ref{lem:ybound}, Lemma~\ref{lem:firstorder}, and Lemma~\ref{lem:secondorder}, we skip the superscript $x_{1:k}$.

\begin{lemma}\label{lem:ybound}

Denoting $\tilde C=C+\log(d+1)+10$, for any $y \in \tilde \Omega_C$ we have that $\tilde w(y) +h(y) \geq 0$, and 
\begin{align}\label{eq:ybound}
&\textstyle e^{-\tilde C}-1 < y_i <e^{\tilde C/2},\, \quad \quad \quad i=1,\dotso, k, \\
&\textstyle e^{-\tilde C/2 } < y_j < 1, \, \quad  \quad \quad \quad \quad  j=k+1, \dotso, d, \notag \\
&\textstyle e^{-\tilde C/2 } < 1-\sum_{l=1}^dT_l(y)<1. \notag
\end{align}
\end{lemma}
\begin{proof}
Take $\tilde D=\{y \in \Omega: \, \tilde w(y)+h(y) < \tilde C\}$. It clear that $\tilde \Omega_C \subset \tilde D$, as $\tilde w -\log(d+1) \leq \tilde \eta \leq \tilde w$.  By the convexity of $r \mapsto r- \log(1+r)$, it is clear that $r -\log(1+r) \geq 0$ for all $r \geq -1$, and $\{r: r -\log(1+r) <M \}$ is an open interval in $\mathbb R$ for any $M >0$.  For any $M>5 $, it can be easily checked that $r-\log(1+r) >M$ for both $r=e^{-2M}-1$ and $r=e^M$, and hence
\begin{align}\label{eq:logbound}
\forall \, M>5,       \quad  r -\log(1+r) < M \implies e^{-2M}-1< r <e^M.
\end{align}
Denote $T_i(y)$ the $i$-th component of the $T(y)$ for $i=1,\dotso, d$. Then for any $y \in \Omega$, it is straightforward that for $i=1,\dotso,k, \   j=k+1,\dotso,d$
\begin{align*}
    \textstyle y_i -\log(1+y_i) \geq 0, \, -\log(y_j) \geq 0 , \,-\log\left(1-\sum_{i=1}^dT_i(y)\right) \geq 0.
\end{align*}
Then according to the definition of $\tilde w$,  for any $y \in \Omega$
\begin{align*}
\textstyle \tilde w(y)+h(y)=  \sum_{i=1}^k 2 (y_i- \log(1+y_i)) -\sum_{i=k+1}^d 2 \log(y_{i})- 2\log\left(1-\sum_{i=1}^dT_i(y)\right)  \geq 0.
\end{align*}
For any $y \in \tilde \Omega_C$, as $\tilde w(y) +h(y) \leq \tilde C$, we get that, for $i=1,\dotso, k,\, j=k+1,\dotso,d$
\begin{align*}
   \textstyle  y_i-\log(1+y_i) \leq \tilde C/2, \, -\log(y_j) \leq \tilde C/2, \, -\log\left(1-\sum_{l=1}^dT_l(y)\right) \leq \tilde C/2,
\end{align*}
and hence from \eqref{eq:logbound} we conclude \eqref{eq:ybound}. 
\end{proof}

In the next lemma, we provide estimates of the first-order and second-order derivatives of $\tilde \eta$. The proofs are of the same flavor as  Lemma~\ref{lem:e1} and Lemma~\ref{lem:e2}, with technical modifications. We provide the complete argument for the convenience of the readers. 

\begin{lemma}\label{lem:firstorder}
Denoting $\tilde C=C+\log(d+1)+10$, it holds that 
\begin{align*}
|\nabla \tilde \eta(y)| \leq 8d^2 \exp \left(5d e^{\tilde C} \right), \quad \forall\, y \in \tilde \Omega_C.
\end{align*}
\end{lemma}
\begin{proof}
The proof will be divided into two steps. Take $\tilde D=\{y \in \Omega: \, \tilde w(y)+h(y) < \tilde C\}$. It clear that $\tilde \Omega_C \subset \tilde D$, as $\tilde w -\log(d+1) \leq \tilde \eta \leq \tilde w$. 
\medskip

\noindent \emph{Step I:} Let us prove the inequality
\begin{align}\label{eq:firstorder}
   \textstyle  |\nabla \tilde \eta (y)| \leq e^{\log(d+1)+h(y)} \sup_{z \in \tilde D} (2+|\nabla \tilde w (z)|) e^{\tilde w(z)}, \quad \forall \, y \in \tilde \Omega_C. 
\end{align}
It is enough to show that for any $\xi \in \R^d$ with $|\xi|=1$, 
\begin{align*}
   \textstyle   \nabla_{\xi} \tilde \eta  \leq e^{\log(d+1)+h(y)} \sup_{z \in \tilde D} (2+|\nabla \tilde w (z)|) e^{\tilde w(z)}, \quad \forall \, y \in \tilde \Omega_C,   
\end{align*}
as $\nabla_{-\xi} \tilde\eta =- \nabla_{\xi} \tilde \eta$. We consider the maximization problem $\max_{y \in \tilde D} \rho(y)$ with 
\begin{align*}
   \rho(y):=(\tilde C-\tilde w(y)-h(y)) \nabla_{\xi} \tilde \eta(y) e^{\tilde \eta (y)}, 
\end{align*}
Since $\rho(y)=0$ on the boundary $\pa \tilde D$, the maximizer must be obtained at the interior of $\tilde D$, denoted by $y_0$, and clearly $\rho(y_0) >0$, $\nabla_{\xi} \tilde \eta (y_0) >0$. The first order condition implies that $0=\nabla_{\xi} \log \rho(y_0)$, i.e., 
\begin{align*}
   \textstyle  0= -\frac{\nabla_{\xi} \tilde w(y_0)+\nabla_{\xi} h(y_0)}{\tilde C-\tilde w(y_0)-h(y_0)}+\frac{\nabla^2_{\xi \xi} \tilde \eta (y_0)}{\nabla_{\xi} \tilde \eta(y_0)}+\nabla_{\xi} \tilde \eta(y_0).
\end{align*}
Thanks to the convexity of $y \mapsto \tilde\eta(y) $, $\nabla^2_{\xi \xi } \tilde \eta(y_0) >0$, and hence 
\begin{align*}
   \textstyle  \nabla_{\xi}\tilde \eta(y_0) < \frac{\nabla_{\xi} \tilde w(y_0)+\nabla_{\xi} h(y_0)}{\tilde C-\tilde w(y_0)-h(y_0)}.
\end{align*}

Therefore we get that 
\begin{align*}
    \rho(y) \leq \rho(y_0) \leq (2+|\nabla \tilde w(y_0)|) e^{\tilde \eta(y_0)}, \quad \forall \, y \in \tilde D. 
\end{align*}
Restricting $y$ in  $\tilde \Omega_C \subset \tilde D$, 
\begin{align*}
 \tilde C-\tilde w(y)-h(y) \geq  C+\log(d+1)+10-\tilde \eta(y)-\log(d+1) -h(y) \geq 10, 
\end{align*}
and also 
\begin{align*}
    \tilde \eta(y) \geq \tilde w(y)-\log(d+1) \geq -h(y) -\log(d+1),
\end{align*}
where in the last inequality we use $\tilde w(y) \geq -h(y)$ from Lemma~\ref{lem:ybound}. 
By the definition of $\rho$, we conclude that
\begin{align*}
    10 \nabla_{\xi} \tilde\eta (y) e^{-h(y)-\log(d+1)} \leq (2+|\nabla \tilde w(y_0)|) e^{\tilde \eta(y_0)}, \quad \forall \, y \in \tilde \Omega_C,
\end{align*}
which implies \eqref{eq:firstorder}.

\medskip
\noindent \emph{Step II:} Let us estimate $e^{\log(d+1)+h(y)} \sup_{z \in \tilde D} (2+|\nabla \tilde w (z)|) e^{\tilde w(z)}$ for $y \in \tilde \Omega_C$.

Due to Lemma~\ref{lem:ybound}, it can easily verified that 
\begin{align*}
e^{\log(d+1)+h(y)} \leq (d+1) \exp\left(d e^{\tilde C} \right), \quad \forall \, y \in \tilde \Omega_C.    
\end{align*}
Next by direct computation 
\begin{align*}
    \partial_{z_i} \tilde w(z)&\textstyle = -\frac{2}{1+z_i}+\frac{2x_i}{1-\sum_{l=1}^d T_l(z)}, \quad i=1,\dotso,k,\\
    \partial_{z_j} \tilde w(z)&\textstyle =-\frac{2}{z_j}+\frac{2}{1-\sum_{l=1}^d T_l(z)}, \quad j=k+1,\dotso, d. 
\end{align*}
Therefore according to Lemma~\ref{lem:ybound}, $|\nabla_z \tilde w(z)| \leq 2d e^{\tilde C}$. Moreover, we have $\tilde w(z) \leq \tilde C-h(z) \leq \tilde C+2d$, and therefore 
\begin{align*}
    e^{\log(d+1)+h(y)} \sup_{z \in \tilde D} (2+|\nabla \tilde w (z)|) e^{\tilde w(z)} \leq (d+1)\exp\left(d e^{\tilde C}\right)\left( 2+2d e^{\tilde C} \right) e^{\tilde C+2d} \leq 8d^2 \exp \left(5d e^{\tilde C} \right).
\end{align*}
\end{proof}

Let us proceed to the second order estimate. By a Pogorelov type interior estimate, we have:  
\begin{lemma}\label{lem:secondorder}
Suppose $C>10$ and denote $\iota=\frac{1}{4e^{C+10}}$. Then we have 
\begin{align*}
    \textstyle \sup_{y \in \tilde \Omega_C, \, \xi \in \mathcal{S}^{d-1}}\left(C-\tilde \eta (y)-h(y)\right)^{2d} \, \nabla^2_{\xi\xi} \tilde\eta \, e^{\frac{ \iota|y|^2}{2}}\leq K(C),
\end{align*}
where $K(C)$ is a positive constant that only depends on $C$. 
\end{lemma}
\begin{proof}
Without loss of generality, we assume the maximizer on the left hand is obtained at $y_0 \in \tilde \Omega_C$, and $\xi=e_1$. Let us denote 
\begin{align*}
  \textstyle   \rho(y)=2d \log(C-\tilde \eta(y)-h(y))+ \log(\tilde \eta_{11}(y))+ \frac{\iota|y|^2}{2},
\end{align*}
and it is straightforward that $y_0 \in \argmax_{y \in \tilde \Omega_C} \rho(y)$. By choosing another orthonormal coordinates, without loss of generality, we assume that $\nabla^2 \tilde\eta(y_0)$ is diagonal, and denote $\lambda_i= \tilde\eta_{ii}(y_0)$ for each $i=1,\dotso, d$. Let us also denote by $(\tilde \eta^{ij})_{i,j=1}^d$ the inverse matrix of $\nabla^2 \tilde\eta(y_0)$. From now on, all the derivatives are taken at $y_0$, and we adopt the Einstein summation notation. The proof will be divided into several steps.

\medskip

\noindent \emph{Step I:} Taking derivatives on both sides of the Monge-Amp\`{e}re $\tilde \eta= \log \det \left(\frac{1}{2} \nabla^2 \tilde \eta \right)$, we have 
\begin{align}\label{eq:2orderMA}
   \teta_i&= \teta^{lk}\teta_{lki}, \quad  \quad\quad\quad\quad\quad\quad \ \ \ i=1,\dotso, d,  \notag \\
   \teta_{ij}&=\teta^{lk}\teta_{lkij}-\teta^{lm}\teta^{kn}\teta_{mnj}\teta_{lki}, \quad i,j=1,\dotso, d.
\end{align}
As $\rho(y_0)=\sup_{y \in \tilde \Omega_C} \rho(y)$, the first order condition implies that 
\begin{align}\label{eq:1order}
   \textstyle  0=\rho_i=-\frac{2d(\teta_i+2)}{C-\teta-h}+\frac{\teta_{11i}}{\teta_{11}}+\iota y_i,  \quad  \quad i=1,\dotso,d.
\end{align}
Taking derivative again, 
\begin{align*}
  \textstyle  \rho_{ij}=-\frac{2d \teta_{ij}}{C-\teta-h}-\frac{2d (\teta_i+2)(\teta_j+2)}{(C-\teta-h)^2}+\frac{\teta_{11ij}}{\lambda_1}-\frac{\teta_{11i}\teta_{11j}}{\lambda_1^2} +\iota \delta_{ij}, \quad \quad i,j=1,\dotso, d . 
\end{align*}
The maximality condition implies that as a matrix $\nabla^2 \rho$ is negative semidefinite, and hence due to the convexity of $\teta$
\begin{align*}
   \textstyle  0 \geq \teta^{ij} \rho_{ij}=- \frac{2d^2}{C-\teta-h}-\sum_{i=1}^d \frac{2d (\teta_i+2)^2 }{\lambda_i (C-\teta-h)^2}+\frac{\teta^{ii }\teta_{11ii}}{\lambda_1}-\sum_{i=1}^d \frac{\teta_{11i}^2}{\lambda_1^2\lambda_i}+\iota\sum_{i=1}^d \frac{1}{\lambda_i }.
\end{align*}

Setting  $i=j=1$ in \eqref{eq:2orderMA}, we get that $\teta^{ii}\teta_{11ii}=\teta_{11}+\teta^{ll}\teta^{kk}\teta_{lk1}^2$, and hence 
\begin{align*}
    0 \geq &\textstyle  - \frac{2d^2}{C-\teta-h}-\sum_{i=1}^d \frac{2d (\teta_i+2)^2 }{\lambda_i (C-\teta-h)^2}+1+\sum_{l,k=1}^d \frac{\teta_{lk1}^2}{\lambda_1\lambda_l \lambda_k}-\sum_{i=1}^d \frac{\teta_{11i}^2}{\lambda_1^2\lambda_i}+\iota\sum_{i=1}^d \frac{1}{\lambda_i } \\
    \geq & \textstyle - \frac{2d^2}{C-\teta-h}-\sum_{i=1}^d \frac{2d (\teta_i+2)^2 }{\lambda_i (C-\teta-h)^2}+1+\sum_{k=2}^d \frac{\teta_{11k}^2}{\lambda_1^2 \lambda_k}+\iota\sum_{i=1}^d \frac{1}{\lambda_i }.
\end{align*}
According to \eqref{eq:1order}, it can be easily seen that for $k=2,\dotso, d$, 
\begin{align*}
    \textstyle -\frac{2d (\teta_k+2)^2 }{\lambda_k (C-\teta-h)^2}=-\frac{1}{2d \lambda_k}\left(\frac{\teta_{11k}}{\lambda_1}+\iota y_k \right)^2 \geq -\frac{1}{d \lambda_k}\left(\frac{\teta_{11k}^2}{\lambda_1^2}+\iota^2 y_k^2 \right).
\end{align*}
Therefore, by direct computation we get that 
\begin{align*}
   \textstyle  \frac{2d^2}{C-\teta-h} +\frac{2d (\teta_1+2)^2 }{\lambda_1 (C-\teta-h)^2} \geq 1+\sum_{k=2}^d \left( \frac{\teta_{11k}^2}{\lambda_1^2 \lambda_k}-\frac{1}{d}\frac{\teta_{11k}^2}{\lambda_1^2 \lambda_k}\right) +\sum_{i=1}^d \frac{\iota-\iota^2y_i^2/d}{\lambda_i}. 
\end{align*}
According to the choice of $\iota= \frac{1}{4e^{C+10}}$ and Lemma~\ref{lem:ybound}, we have 
\begin{align*}
    \textstyle \iota^2 y_i^2 /d \leq  \frac{\iota  (d+1)e^{C+10}}{4de^{C+10}}  \leq \frac{\iota}{2},
\end{align*}
and therefore 
\begin{align}\label{eq:step1}\textstyle 
\frac{2d^2}{C-\teta-h} +\frac{2d (\teta_1+2)^2 }{\lambda_1 (C-\teta-h)^2} \geq \sum_{i=1}^d \frac{\iota}{2\lambda_i}.
\end{align}

\medskip

\noindent \emph{Step II:} Let us prove a lower bound for $\lambda_1$ using the Monge-Amp\`{e}re equation. Note that $$\textstyle \lambda_1^d \geq \lambda_1 \dotso \lambda_d=\det \nabla^2\tilde \eta(y_0)=2^d e^{\tilde \eta}\geq 2^d e^{\tilde w - \log(d+1)}=\frac{2^d}{d+1} e^{\tilde w}.$$
According to the definition, it can be easily seen that 
\begin{align*}
  \textstyle   \tilde w(y)= -2 \sum_{i=1}^k \log(1+y_i) -2 \sum_{i=k+1}^d \log(y_i)-2 \log \left(1 -\sum_{i=1}^d T_i(y) \right), 
\end{align*}
and hence according to Lemma~\ref{lem:ybound}, $\tilde w$ has a lower bound only depends on $C$. From now on, we denote by $K(C)$ a constant that only depends on $C$ and it is allowed to change from line to line. Hence we conclude that $\lambda_1 \geq K(C)$, and together with \eqref{eq:step1},
\begin{align*}\textstyle 
\frac{1}{C-\teta-h} +\frac{ (\teta_1+2)^2 }{(C-\teta-h)^2} \geq K(C) \sum_{i=1}^d \frac{1}{\lambda_i}.
\end{align*}

Now, we would like to derive an upper bound of $\lambda_1$,
\begin{align*}
  \textstyle   \sum_{i=1}^d \frac{1}{\lambda_i} \geq \sum_{i=2}^d \frac{1}{\lambda_i} \geq (d-1) \sqrt[d-1]{\frac{1}{\lambda_2 \dotso \lambda_d }}=(d-1) \sqrt[d-1]{\frac{\lambda_1}{\det \nabla^2 \tilde \eta }} =(d-1) \sqrt[d-1]{\frac{\lambda_1}{2^d e^{\tilde \eta} }}.
\end{align*}
Since $y \in \tilde \Omega_C$, we have $e^{\tilde \eta(y)} \leq  e^{C-h(y)} \leq e^{C+2d}$ according to Lemma~\ref{lem:ybound}. Therefore it can be easily seen that 
\begin{align*}
   \textstyle \lambda_1 \leq K(C) \left( \sum_{i=1}^d \frac{1}{\lambda_i} \right)^{d-1}\leq K(C) \left(\frac{1}{C-\teta-h} +\frac{ (\teta_1+2)^2 }{(C-\teta-h)^2} \right)^{d-1},
\end{align*}
and hence together with Lemma~\ref{lem:ybound} and Lemma~\ref{lem:firstorder}
\begin{align*}
    \left(C-\tilde \eta (y_0)-h(y_0)\right)^{2d} \, \lambda_1 \, e^{\frac{ \iota|y_0|^2}{2}} \leq K(C) \left(1+ (\tilde \eta_1+2)^{2d-2}\right) \leq K(C).
\end{align*}

\end{proof}

\begin{lemma}\label{lem:fsecondorder}
     Suppose $x^0 \in \partial \Delta$ with $x^0_1=\dotso=x^0_k=0$, $x^0_{k+1},\dotso,x^0_d >0$, and $\sum_{j=k+1}^d x_j^0<1$. Taking $r= (1-\sum_{i=k+1}^d x^0_i)/{(2d)}$,  there exist two positive constants $c_1(x^0)$, $c_2(x^0)$ such that for any $x_1,\dotso, x_k \in (0,r)$,
    \begin{align*}
   c_1(x^0) I_d \leq \nabla^2 \tilde \eta^{x_{1:k}}(x^0) \leq c_2(x^0) I_d,
  \end{align*}
  where $\tilde\eta^{x_{1:k}}$ is defined in \eqref{eq:tileta}.
\end{lemma}
\begin{proof}
 We note that for $x_1,\dotso,x_k \in (0,r)$,
\begin{align*}
    \tilde \eta^{x_{1:k}}(x^0) \leq &\textstyle  \tilde w^{x_{1:k}}(x^0)= -2 \sum_{i=k+1}^d \log(x^0_i)-2 \log \left(1-\sum_{i=1}^k x_1 - \sum_{i=k+1}^d x^0_i \right) \\
    \leq &\textstyle  -2 \sum_{i=k+1}^d \log(x^0_i)-2 \log \left(\frac{1- \sum_{i=k+1}^d x^0_i }{2} \right).
\end{align*}
Taking $C:= -2 \sum_{i=k+1}^d \log(x^0_i)-2 \log \left(\frac{1- \sum_{i=k+1}^d x^0_i }{2} \right)+10 $, we apply Lemma~\ref{lem:secondorder} to get 
\begin{align*}
   \textstyle  \left(C- \tilde \eta^{x_{1:k}}(x^0)\right)^{2d} \lambda_{max} \left(\nabla^2 \tilde \eta^{x_{1:k}} (x^0)\right) \leq K(C).
\end{align*}
According to the choice of $C$, it is straightforward that $\left(C-\tilde\eta^{x_{1:k}}(x^0) \right)^{2d} \geq 10^{2d}$. Consequently, we obtain that 
\begin{align*}
   \textstyle  \lambda_{max} \left(\nabla^2 \tilde \eta (y^0)\right) \leq \frac{K(C)}{10^{2d}}=:c_2,
\end{align*}
and $c_2$ is a constant that only depends on $x^0_{k+1},\dotso, x^0_d$.

To conclude, according to the Monge-Amp\`{e}re equation, we have 
\begin{align*}
    \textstyle \lambda_{min} \left(\nabla^2 \tilde \eta^{x_{1:k}} (x^0)\right) c_2^{d-1} \geq \det \nabla^2 \tilde \eta^{x_{1:k}} (x^0)= 2^d e^{\tilde\eta^{x_{1:k}}(x^0)}.
\end{align*}
Due to the inequality $|\tilde \eta^{x_{1:k}}-\tilde w^{x_{1:k}}| \leq \log (d+1)$,$$\textstyle \tilde\eta^{x_{1:k}}(x^0) \geq \tilde w^{x_{1:k}}(x^0)- \log(d+1)\geq  -2 \sum_{i=k+1}^d \log(x^0_i)-2 \log \left(1 - \sum_{i=k+1}^d x^0_i \right)-\log(d+1),$$
and hence 
\begin{align*}
     \lambda_{min} \left(\nabla^2 \tilde \eta^{x_{1:k}} (x^0)\right) \geq c_1,
\end{align*}
where $c_1$ is a constant only depends on $x^0_{k+1},\dotso, x^0_d$.
\end{proof}

\begin{proof}[Proof of Proposition~\ref{prop:secondorder}]
It can be easily seen that $ \nabla^2 \eta^{x_{1:k}}(y)=\nabla^2 \tilde \eta^{x_{1:k}}(y)$. Denoting $D= Diag(1/x_1,\dotso,1/x_d, 1, \dotso, 1)$, due to \eqref{eq:htransfer} we obtain the equality $$\nabla^2 g(x_{1:k},x^0_{k+1:d})= D \, \nabla^2 \tilde \eta^{x_{1:k}}( (T^{x_{1:k}})^{-1}(x_{1:k},x^0_{k+1:d})) \, D, $$
Therefore the result follows from  Lemma~\ref{lem:fsecondorder} and the fact that $(T^{x_{1:k}})^{-1}(x_{1:k}, x^0_{k+1:d})=x^0$ 
\end{proof}

\begin{proof}[Proof of Theorem~\ref{thm:quantitive}]

Part \emph{(i)} is a direct corollary of Proposition~\ref{prop:secondorder}. Let us prove the second claim of the theorem.  We define the sets, for $i=1,\dotso,d$,
    \begin{align*}
    E_i=&\textstyle  \left\{x \in \Delta: \, x_i \geq \frac{1}{d+1} \right\}, \\
    E_0=& \textstyle \left\{x \in \Delta: \, \sum_{i=1}^d x_i \leq \frac{d}{d+1} \right\}.
    \end{align*}
It is clear that $E_0 \cup E_1 \cup \dotso \cup E_d=\Delta$. We prove Theorem~\ref{thm:quantitive} (ii) on each $E_i$, $i=0,\dotso, d$. 

\medskip 

\noindent \emph{Step I:} It is sufficient to show the result on $E_0$. Indeed, if $x \in E_i$, we construct a linear transformation $L: \Delta \to \Delta$ through 
\begin{align*}
    y_i&=L_i(x)=1- x_1-\dotso- x_d, \\
    y_j&=L_j(x)=x_j, \hspace{3cm} j \in \{1,\dotso, d\} \setminus \{i\}. 
\end{align*}
Then if $x \in E_i$ with $x_i \geq \frac{1}{d+1}$, one gets immediately that $\sum_{i=1}^d L_j(x)=1-x_i \leq \frac{d}{d+1}$ and thus $L(x) \in E_0$. Moreover, $L=L^{-1}$ and the determinant of the Jacobian of $L$ equals to $-1$. Therefore $\bar g(x):= g(L^{-1}(x))$ satisfies \eqref{eq:Monge} as well. 
\medskip

\noindent \emph{Step II:} Let us prove Theorem~\ref{thm:quantitive} (ii) on $E_0$. For any $x \in E_0$, as before we define 
\begin{align*}
    T^x(y):=& (x_1(1+y_1),x_2(1+y_2),\dotso, x_d(1+y_d)), \quad y \in \Omega^x:=(T^x)^{-1}(E_0) \\
    \tilde\eta^x(y):=&\textstyle g(T^x(y))+2\sum_{i=1}^d \log(x_i), \\
    \tilde w^x(y):=&\textstyle w(T^x(y))+2\sum_{i=1}^d \log(x_i). 
\end{align*}
By direct computation, it can be easily seen that
\begin{align*}
    \nabla_{y_i} \tilde\eta^x (0)=x_i \nabla_{x_i}g(x), \quad \nabla^2_{y_iy_j} \tilde \eta^x (0)=x_i x_j \nabla^2_{x_ix_j}g(x), \quad i,j=1,\dotso, d.
\end{align*}
Therefore, we have the equality
\begin{align*}
     \nabla g(x)^{\top} (\nabla^2 g(x))^{-1} \nabla g(x)= \nabla \tilde \eta^x(0)^{\top} (\nabla^2 \tilde \eta^x (0))^{-1} \nabla \tilde \eta^x(0),
\end{align*}
and the result directly follows from Lemma~\ref{lem:firstorder} and Lemma~\ref{lem:fsecondorder}.
    
\end{proof}

\section{The multidimensional Aldous martingale}\label{sec6}
In this section, we show that \eqref{eq:AMTG} is the unique optimizer to \eqref{eq:optimal}. Let us define \[\textstyle \sigma_*(t,x):=\sqrt{\frac{2 (\nabla^2 g(x))^{-1} }{1-t}}.\]
Given $x\in \Delta$, we will first argue that \eqref{eq:AMTG}, i.e., 
  \begin{align}\label{eq:Ald_final}
  dM_t=\sigma_*(t,M_t) \, d B_t, \,\, M_0=x, 
  \end{align}
admits a unique strong solution  up to $1 \wedge \tau$, where $\tau:=\inf\{t\in[0,1]: M_t\notin\Delta\}$ denotes the first exit time of $\Delta$.
 By the deterministic time-change  $Y_t:=M_{1-e^{-t/2}}$, it can be easily verified that  
    \[\textstyle  dY_t= \sqrt{(\nabla^2g(Y_t))^{-1}}\,dW_t, \]    
    for a $d$-Brownian motion $(W_t)_{t \in\mathbb R_+}$. Thus we can w.l.o.g.\ switch between $Y$ and $M$ as needed.\\

\begin{lemma}\label{lem:exists_aldous}
    There exists a unique strong solution to \eqref{eq:Ald_final}  over $[0, 1\wedge \tau)$. 
\end{lemma}    
\begin{proof}
    We pass to the SDE for $Y$. We let $\Delta^n$ be an increasing sequence of open subsets of $\Delta$ such that $\cup_n \Delta^n=\Delta$ and $\overline{\Delta^n}\subset \Delta$. By smoothness of $g$, and the fact that $\nabla^2g$ is non-singular throughout $\Delta$, we know that $\overline{\Delta^n}\ni x\mapsto\sqrt{(\nabla^2 g(x))^{-1}}$ is Lipschitz.

\medskip

    \noindent \emph{Step I:} We prove that there is at most one process $Y$ adapted to $W$, with $Y_0=x$, and such that $dY_t= \sqrt{(\nabla^2g(Y_t))^{-1}}\,dW_t$ for $t<\eta_\infty:=\inf\{s:Y_s\notin \Delta\}$.

    To wit, suppose that $Y^1,Y^2$ with these properties exist, and denote by $\eta_\infty^1,\eta_\infty^2$ the respective exit times. Also denote $\eta_n^i:=\inf \{s:Y^1_s\notin \Delta^n\}$, for $i\in\{1,2\}$, so $\eta_n^i\nearrow \eta_\infty^i$. For every $n$, we let $\sigma_n: \mathbb R^d \to \mathbb S^d$ be any Lipschitz function equal to $\sqrt{(\nabla^2 g(x))^{-1}}$ on $\overline{\Delta^n}$, its existence being guaranteed e.g.\ by the Kirszbraun theorem. Then for each $n$, the SDE $dZ_t=\sigma_n(Z_t) \,dW_t$ admits exactly one strong global solution starting from any initial condition. For $i\in\{1,2\}$, we can introduce the process $Y^{i,n}$ as follows: $Y^{i,n}_t:= Y^i_t$ for $t\leq\eta^i_n$, and for $t>\eta^i_n$ we define $Y^{i,n}_t$ as the value of $Z_{t-\eta^i_n}$, where $Z$ has been started at the position $Y^i_{\eta^i_n}$ and is being driven by the Brownian motion $(W_{s+\eta^i_n}-W_{\eta^i_n})_{s\in\mathbb R_+}$. Hence $Y^{i,n}$ fulfills the SDE for $Z$ with $Y^{i,n}=x$. By strong uniqueness, we conclude  $Y^{1,n}=Y^{2,n}$. In particular then $\eta^1_n=\eta^2_n=:\eta_n$, and $(Y^1_s)_{s\in [0,\eta_n]}=(Y^2_s)_{s\in [0,\eta_n]}$. Sending $n\to\infty$ we conclude that $\eta^1_\infty=\eta^2_\infty=:\eta_\infty$ and $(Y^1_s)_{s\in [0,\eta_\infty]}=(Y^2_s)_{s\in [0,\eta_\infty]}$.
\medskip
   
    \noindent\emph{Step II:}  We prove that there is a process $Y$ adapted to $W$, with $Y_0=x$, and such that $dY_t= \sqrt{(\nabla^2g(Y_t))^{-1}}\,dW_t$ for $t<\eta_\infty:=\inf\{s:Y_s\notin \Delta\}$.

For every $n$, we let $\sigma_n: \mathbb R^d \to \mathbb S^d$ be any Lipschitz function equal to $\sqrt{(\nabla^2 g(x))^{-1}}$ on $\overline{\Delta^n}$. We can define $Y$ inductively. First we introduce $Y^{(1)}$ as the unique strong solution to $dY_t^{(1)}=\sigma_1(Y_t^{(1)}) \,dW_t$ with $Y_0^{(1)}=x$. Defining $\eta_1=\inf\{s: Y_s^{(1)}\notin\Delta^1\}$, we set $Y_t= Y_t^{(1)} $ for $t\leq \eta_1$. Observe that $\eta_1=\inf\{s: Y_s\notin\Delta^1\}$ and that $dY_t= \sqrt{(\nabla^2 g(Y_t))^{-1}} \,dW_t$ for $t<\eta_1$. We can now extend the hitherto constructed process $Y$ to times $t>\eta_1$ by concatenating it with the solution of $dY_s^{(2)}=\sigma_2(Y_s^{(2)}) \,d\tilde W_s$ with $Y_0^{(2)}=Y_{\eta_1}$, where $\tilde W_s= W_{s+\eta_1}-W_{\eta_1}$. Calling $\eta_2$ the first exit time of $Y^{(2)}$ from $\Delta^2$, we can as before argue that $dY_t= \sqrt{(\nabla^2 g(Y_t))^{-1}}dW_t $ for $t\leq \eta_1+\eta_2 = \inf\{s:Y_s\notin\Delta^2\}$. Going on like this by induction, we conclude.
    
\end{proof}

In fact, by martingale convergence, { existence and} uniqueness in Lemma \ref{lem:exists_aldous} extends to the closure of the stochastic interval therein.  The proof of the next result is a Lyapunov argument based on \cite[Theorem 3.5]{Kh12}.

\begin{lemma}\label{lem:exit}
   Denoting $\hat \tau= \inf \{t \in [0,1]: M_t \not \in \Delta \}\wedge 1$, we have $\hat \tau=1$ almost surely. 
\end{lemma}
\begin{proof}
Passing to $Y$ it is clear that ``$\hat \tau=1$ a.s.'' is equivalent to ``$\eta=\infty$ a.s.'', where
    \begin{align*}
        \eta:=\inf\{t \geq 0: \, Y_t \not \in \Delta \}.
    \end{align*}

    For any $n \in \mathbb N$, take $\Delta_n:= \{x \in \Delta: \, g(x) < n \}$, and $\eta_n$ to be the exit time of $(Y_t)_{t \geq 0}$ from $\Delta_n$. As $ \cup_{n \geq 1}\Delta_n=\Delta$, and the volatility $\sqrt{(\nabla^2g(x))^{-1}}$ is a smooth function over $\bar\Delta_n$, the SDE of $(Y_t)_{t \geq 0}$ is well-defined over $[0,\eta_n]$.  As $g \geq w- \log(d+1) >0$ over $\Delta$ and $\inf_{x \in \Delta} \left( w(x)-\log(d+1)  \right)>0$ due to Lemma~\ref{lem:maincomponent}, we get that 
    \begin{align*}
       \textstyle  \frac{1}{2} \sum_{i,j=1}^d \nabla^2_{ij}g(x) [\nabla^2g(x)]^{-1}_{ij}=\frac{1}{2}\Tr\left(\nabla^2g(x)[\nabla^2g(x)]^{-1} \right ) =\frac{d}{2} \leq c g(x), \quad \forall \, x \in \Delta
    \end{align*}
    for a large enough positive constant $c$. A direct application of It\^{o}'s formula  shows that  $e^{-ct}g(Y_t)$ is super-martingale over $[0,\eta_n]$ for any $n \in \mathbb{N}$, and hence 
    \begin{align*}
       \E[g(Y_0)] \geq  \E[e^{-c(t \wedge \eta_n)} g(Y_{t \wedge \eta_n})] .
    \end{align*}

We prove that $\eta > t$ \emph{a.s.}\ for all $t >0$, from which it follows that  $\eta=\infty$ \emph{a.s.} Denote $A:=\{ \omega: \, \eta(\omega) \leq t \}$. It is clear $\eta_n \leq t$ over $A$, and 
\begin{align*}
   \E[g(Y_0)] \geq \E[e^{-c(t \wedge \eta_n)} g(Y_{t \wedge \eta_n})] \geq \E[\mathbbm{1}_A e^{-ct} g(Y_{\eta_n}) ]=\mathbb P(A) e^{-ct}n,  \quad  \forall n \in \mathbb N. 
\end{align*}
Therefore we get that $\mathbb P(A)=0$, which finishes the proof of this result. 
\end{proof}

Using Theorem~\ref{thm:quantitive}, we show that $(M_s)$ is indeed a win-martingale. 
\begin{lemma}\label{lem:appl_MA_win}
    $M_1\in \{e_0,e_1,\dots,e_d\}$ a.s.
\end{lemma}

\begin{proof}
    In light of Lemma \ref{lem:exit},  we have an a.s.\ dichotomy: either $M_1\in \Delta$, or { $M_1\in \partial \Delta$ while $M_s\in\Delta$ for all $s<1$.}

 As the martingale $M$ is bounded, we have by the BDG inequality that $\mathbb E[\langle M^i,M^k \rangle_1]<\infty$ for each $i,k$. Hence
    \begin{align*}\textstyle 
\mathbb E \left[ \int_0^1 \Tr \left(\sigma_{*} \sigma^{\top}_{*}(t,M_t)  \right)\, dt\right]= \mathbb E \left[ \int_0^1 \frac{2 \Tr  \left( (\nabla^2 g(M_t))^{-1} \right)}{1-t}  \, dt\right] < +\infty. 
\end{align*}

On the event $\{M_1\in\Delta\}$ we have a.s., by martingale convergence, that $$\textstyle \int_0^1 \frac{2 \Tr  \left( (\nabla^2 g(M_t))^{-1} \right)}{1-t}  \, dt = \infty,$$
as the numerator of the integrand concentrates around a strictly positive value for $t\sim 1$. 
Hence $\{M_1\in\Delta\}$ must have measure zero. On the other hand,
observe that $\Tr   (\nabla^2 g(M_t)^{-1}) \geq \frac{1}{\lambda_{min}(\nabla^2 g(M_t)) }$ by the positivity of eigenvalues.  Thanks to Theorem \ref{thm:quantitive} (i), we have that $$\liminf_{t\to 1} \Tr  \left( (\nabla^2 g(M_t))^{-1} \right)>0,$$ a.s.\ on   $M_1\in  \partial\Delta \backslash \{e_0,e_1,\dots,e_d\}  $. But then again $\int_0^1 \frac{2 \Tr  \left( (\nabla^2 g(M_t))^{-1} \right)}{1-t}  \, dt = \infty$ a.s.\ on   $M_1\in  \partial\Delta \backslash \{e_0,e_1,\dots,e_d\}  $. We conclude that the latter must be a measure zero set too. 
\end{proof}

Denote $\Sigma_*(t,x)=\sigma_*\sigma_*^\top (t,x)=(\sigma_*(t,x))^2 =\frac{2 (\nabla^2 g(x))^{-1} }{1-t}$, and recall $f(t)=d(1-t)\log(1-t)$. From the following result, we deduce that the $(M_t)$ realizes the optimal value \eqref{eq:optimal}.

\begin{lemma}\label{lem:mart_prop_logsigma}
    $\log\det\Sigma_*(t,M_t)  $ is a martingale on $[0,1)$,  and we have
    \[\textstyle \mathbb E\left[\int_t^1 -\log\det\Sigma_*(u,M_u) \, du \,\Big|\,\mathcal F_t  \right] = f(t)+(1-t)g(M_t) .\]
\end{lemma}

\begin{proof}

As $g$ is the smooth solution to \eqref{eq:Monge}, we get that
\begin{align*}\textstyle 
\log \det \Sigma_*(t,M_t)&\textstyle =-d\log(1-t)-\log\det \left(\frac{1}{2}\nabla^2 g(M_t) \right) =-d \log (1-t) -g(M_t). 
\end{align*}
Then an application of It\^{o}'s formula yields that 
\begin{align*}
    d \log \det \Sigma_*(t,M_t)=&\textstyle  \frac{d}{1-t} dt - \frac{1}{2} \sum_{i,j=1}^d \nabla^2_{ij} g(M_t)   d \langle M^i,M^j \rangle_t - \nabla g(M_t) \,dM_t \\
    =&\textstyle \frac{d}{1-t} dt-\frac{1}{2} \Tr\left(\nabla^2 g(M_t) \frac{2 (\nabla^2 g(M_t))^{-1}}{1-t} \right) dt -\nabla g(M_t) \,dM_t \\
    =& -\nabla g(M_t) \, dM_t,
\end{align*}
and therefore $\log \det \Sigma_*(t,M_t)$ is a local martingale. 

 Due to the SDE \eqref{eq:Ald_final}, the instantaneous quadratic variation of $\log \det \Sigma_*(t,M_t)$ is $$\textstyle \frac{2}{1-t}\nabla g(M_t)^\top (\nabla^2 g(M_t))^{-1} \nabla g(M_t)\leq \frac{2}{1-t}\sup_{z\in \Delta}\nabla g(z)^\top (\nabla^2 g(z))^{-1} \nabla g(z) ,$$ and according to Theorem \ref{thm:quantitive}.(ii)
the supremum on the r.h.s.\ is finite. Hence $\log \det \Sigma_*(s,M_s)$ is a martingale for $s\in[0,1)$. For convenience of the reader we provide an alternative argument which is not based on Theorem \ref{thm:quantitive}: 
Consider the localization of $M_t^{\tau_n}$ for $n  \in \mathbb N$, where $\tau_n:=\inf \left\{t: \, g(M_t) \geq n \right\}$. According to Lemma~\ref{lem:exit} and Lemma~\ref{lem:appl_MA_win}, we have $\lim\limits_{n \to \infty} \tau_n=1$. Then it is straightforward that $$\left(g(M^{\tau_n}_t)+d\log(1-t \wedge \tau_n)\right)_{t \in [0,1)}$$ is a martingale. As $g\geq 0$, Fatou's lemma gives $$\E[g(M_t)+d \log(1-t)] \leq \liminf_{n \to \infty} \E \left[g(M^{\tau_n}_t)+d\log(1-t \wedge \tau_n)\right]<+\infty.$$
Therefore for any $t \in [0,1)$, $\E[g(M_t)] < +\infty$. Then due to Jensen's inequality $$0\leq g(M_t^{\tau_n}) \leq \E[g(M_t) \,| \, \mathcal{F}_{t \wedge \tau_n}],$$ which implies that  $\{g(M_t^{\tau_n}): \, n  \in \mathbb N \}$ is uniformly integrable as it is bounded from above by conditional expectations of the integrable random variable $g(M_t)$. Hence the almost surely convergence  $g(M^{\tau_n}_t)+d\log(1-t \wedge \tau_n) \to g(M_t)+d\log(1-t)$ holds in $L^1$, for any $t \in [0,1)$. We conclude that $g(M_t)+d\log(1-t)=- \log \det \Sigma_*(t,M_t)$ is also a martingale.

For the second claim, we have
\begin{align*}
     &\textstyle \mathbb E\left[\int_t^{1} -\log\det\Sigma_*(u,M_u) \,du \,\Big|\,\mathcal F_t  \right]  = \mathbb E\left[ \int_t^1 d \log(1-u)+g(M_u) \,du \,\Big|\,\mathcal F_t  \right] \\
     &\textstyle =\lim\limits_{\epsilon \to 0}\mathbb E\left[ \int_t^{1-\epsilon} d \log(1-u) \, du\,\Big|\,\mathcal F_t  \right]+ \lim\limits_{\epsilon \to 0}\mathbb E\left[ \int_t^{1-\epsilon} g(M_u) \, du\,\Big|\,\mathcal F_t  \right] \\
     &\textstyle =\lim\limits_{\epsilon \to 0}\mathbb E\left[ \int_t^{1-\epsilon} d \log(1-u)+g(M_u) \,du\,\Big|\,\mathcal F_t  \right] =(1-t) d \log(1-t)+ (1-t)g(M_t),
\end{align*}
where we make use of the conditional monotone convergence theorem in the second equality, and the martingale property of $d\log(1-u)+g(M_u)$ for $u<1$ in the last equality. 
\end{proof}
 
We can conclude the proof of Theorem \ref{thm:intro_existence}.
 
\begin{proof}[Proof of Theorem \ref{thm:intro_existence}]
  By Lemmas \ref{lem:exists_aldous}-\ref{lem:exit} the Aldous martingale exists, is uniquely determined as a strong solution  to \eqref{eq:AMTG}  up to time 1, and by Lemma \ref{lem:appl_MA_win} it is a win-martingale. By Lemma \ref{lem:mart_prop_logsigma} we have
    $f(t)+(1-t)g(x)=\mathbb E\left[\int_t^1 -\log\det \Sigma_*(u,M^{t,x}_u)\,du,\right]$, 
    for $M^{t,x}$ the Aldous martingale started at $x$ at time $t$. As shown in Corollary \ref{cor:val_function_fg}, the value function of the problem is $f(t)+(1-t)g(x)$.
    Taking $t=0$ we conclude the optimality of the Aldous martingale. For the uniqueness of the optimizer we can argue just as in \cite{BBexciting}.
\end{proof}

\bibliographystyle{siam}
\bibliography{ref.bib}

\end{document}